\documentclass[12pt]{amsart}

\usepackage{times,amsfonts,amsmath,amstext,amsbsy,amssymb,amsopn,amsthm,upref,eucal,amscd}
\usepackage[T1]{fontenc}
\usepackage[all]{xy}
\usepackage[pdftex]{graphicx}
\usepackage{color}
\usepackage[hyphens]{url}

\newtheorem{theorem}{Theorem}[section]

\newtheorem{proposition}[theorem]{Proposition}

\numberwithin{equation}{section}
\theoremstyle{definition}

\newtheorem{example}[theorem]{Example}

\newtheorem{remark}[theorem]{Remark}

 \def\R{\mathbb{R}}
 \def\Z{\mathbb{Z}}

\def\leq{\leqslant }
\def\geq{\geqslant}

\def\SU{\mathrm{SU}}
\def\SO{\mathrm{SO}}
\def\SL{\mathrm{SL}}
\def\U{\mathrm{U}}
\def\C{\mathbb{C}}
\def\hm{\mathrm{Hom}}
\def\X{\mathfrak{X}}
\def\Y{\mathfrak{Y}}
\def\tr{\mathrm{tr}}

\DeclareSymbolFont{bbold}{U}{bbold}{m}{n}
\DeclareSymbolFontAlphabet{\mathbbold}{bbold}
\def\bbone{\mathbbold{1}}

\begin{document}

\title[Non-ergodicity on character varieties]{Non-ergodicity on $\SU(2)$ and $\SU(3)$ character varieties of the once-punctured torus}

\author[G. Forni]{Giovanni Forni}
\address{Department of Mathematics, University of Maryland, College Park, MD 20742, USA}
\email{gforni@umd.edu}

\author[W. Goldman]{William M. Goldman}
\address{Department of Mathematics, University of Maryland, College Park, MD 20742, USA}
\email{wmg@math.umd.edu}

\author[S. Lawton]{Sean Lawton}
\address{Department of Mathematical Sciences, George Mason University, 4400 University Drive, Fairfax, Virginia 22030, USA}
\email{slawton3@gmu.edu}

\author[C. Matheus]{Carlos Matheus}
\address{CMLS, CNRS, \'Ecole polytechnique, Institut Polytechnique de Paris, 91128 Palaiseau Cedex, France.}
\email{carlos.matheus@math.cnrs.fr.}

\subjclass[2010]{Primary 14M35, 22D40, 70H08; Secondary 53D30, 37A25}


\date{\today}
\keywords{KAM Theory, Non-ergodicity, Character Variety}

\begin{abstract}
Utilizing KAM theory, we show that there are certain levels in relative $\SU(2)$ and $\SU(3)$ character varieties of the once-punctured torus where the action of a single hyperbolic element is not ergodic.
\end{abstract}
\maketitle


\tableofcontents

\section{Introduction}

In 1997 Goldman proved that the (pure) mapping class group of a closed orientable surface $\Sigma$ acts ergodically on the moduli space of flat principal $G$-bundles over $\Sigma$ when $G$ is a compact connected Lie group whose simple factors are rank 1 \cite{Gold6}.  Goldman conjectured his theorem generalized to parabolic $G$-bundles over compact orientable surfaces $\Sigma_{n,g}$ of genus $g\geq 0$ with $n\geq 0$ punctures (assuming $n\geq 4$ if $g=0$) for any compact connected Lie group $G$.  This conjecture was established in \cite{PE1,PE2} when $g\geq 2$ and $n\geq 0$. The genus 0 and 1 cases are largely open\footnote{When $g=0$ the conjecture is trivially true for $n=0,1,2$ since the moduli spaces are points, and false for $n=3$  since the mapping class group is trivial unless $G$ has only rank 1 simple factors when again the moduli space is a point.}  when $G$ has simple factors of rank {\it greater} than 1.  However, when $G=\SU(3)$ and $g=1=n$ the conjecture was recently shown true \cite{GLX}.  All existing proofs more or less show that the Dehn twists are covered by Hamiltonian flows which generate an ergodic group action.  Recent work has been announced that the much smaller Torelli subgroup is sufficient to establish ergodicity in the closed surface cases \cite{Bo}.

It is natural to ask (as with the action of an irrational rotation on a torus) if a single hyperbolic element in the mapping class group is sufficient to induce ergodicity.    In the case of $G=\SU(2)$, it was shown in Brown's thesis \cite{Bro} that the answer is no using KAM theory.  There is a small gap in his proof that was recently observed by C. Matheus.  It is the purpose of this article to fill this gap and further establish this phenomena also occurs in the case of $G=\SU(3)$ and $g=1=n$ giving the opportunity to utilize higher dimensional KAM theory.

In addition to giving a complete account of Brown's theorem, we prove:

\begin{theorem}[Theorem \ref{main}]Let $\mathcal{MCG}$ be the mapping class group of the once-punctured torus $\Sigma_{1,1}$ and let $\X_\ell$ be the moduli space of flat parabolic $\SU(3)$-bundles over $\Sigma_{1,1}$ with parabolic data $\ell$. Then, there exists a hyperbolic element $M\in\mathcal{MCG}$ and a choice of parabolic structure $\ell$ such that the action of $M$ on $\X_\ell$ is non-ergodic. 
\end{theorem}

The theorem is proved by brute force.  Knowing the explicit structure of $\X_\ell$ from work of Lawton \cite{La1, FlLa1} we computationally demonstrate that the conditions of KAM theory are satisfied.  There are many subtle points to this computation which we highlight along the way.

It is natural to ask if we can generalize this work to $\SU(n)$ for $n\geq 4$.  Unfortunately, we lack a sufficient computational understanding of the character varieties in these more general cases to carry out our proof.

The rest of the article is organized as follows.  In Section \ref{sec-2} we review the moduli space of flat parabolic bundles from the point of view of character varieties. In Section \ref{sec-3} we discuss the symplectic geometry of character varieties setting up the context where KAM theory can be employed.  In Section \ref{sec-4} we give a complete treatment of Brown's theorem filling in a gap in his original proof.  This is accomplished by invoking a theorem of R\"ussmann to make up for the gap.   Section \ref{sec-5} concerns fixed points of the cat map (our choice of hyperbolic element of the mapping class group) in preparation for the proof of our main theorem.  In Section \ref{sec-6}, we first illustrate the proof of our main theorem in the context of Brown's theorem.  This allows the reader to ``warm up'' to the ideas that will be used to prove our main theorem (which is significantly more complicated).  Thereafter, we prove the main theorem, outsourcing computations to {\it Mathematica} notebooks we have made public on GitHub\footnote{\url{https://github.com/seanlawton/Non-ergodicity-on-SU-2-and-SU-3-character-varieties-of-the-once-punctured-torus}}.  The computations depend on precise formulations from KAM theory in the context of symplectomorphisms.  As these formulations are not explicitly in the literature, we put them in Appendices \ref{appa} and \ref{appb}.

\subsection*{Acknowledgments}
Lawton is partially supported by a Collaboration grant from the Simons Foundation, and thanks IHES for hosting him in 2021 when this work was advanced.  Forni was supported by the NSF grant DMS 2154208, and Goldman was supported by NSF grant DMS 1709791.  We thank Rapha\"el Krikorian and Peter Gothen for helpful correspondence.  We also thank both referees for helpful remarks.

\section{Character Varieties of the Punctured Torus}
\label{sec-2}

\subsection{Character Varieties} 
Let $G$ be a real reductive Lie group, and let $\Gamma$ be a finitely presentable group.  The set of homomorphisms $\hm(\Gamma, G)$ admits a natural topology from $\Gamma$ and $G$ as follows.  As $\Gamma$ is finitely presentable, it admits a presentation $\Gamma_P=\langle \gamma_1,...,\gamma_r\ |\ w_1,...,w_s\rangle$ where the $w_i$'s are words in the generators.  Then, the function $\iota_P:\hm(\Gamma_P,G)\to G^r$ given by $\rho\mapsto (\rho(\gamma_1),...,\rho(\gamma_r))$ is injective and so we declare a set $U\subset  \hm(\Gamma_P,G)$ to be open if and only if $\iota_P(U)$ is open in $\iota_P(\hm(\Gamma_P,G))\subset G^r$ with the subspace topology.  In fact, letting $\bbone\in G$ be the identity, $\hm(\Gamma_P,G)$ is cut out of the manifold $G^r$ by the analytic equations $\{w_i(g_1,...g_r)=\bbone\ |\ 1\leq i\leq s\}$ making $\hm(\Gamma_P,G)$ an {\it analytic subvariety} of $G^r$. 

Now if $\Gamma_{P'}=\langle \gamma_1',...,\gamma_{r'}'\ | w_1',...,w_{r'}'\rangle$ is a different presentation of $\Gamma$, there exists an isomorphism of groups $\Phi:\Gamma_P\to \Gamma_{P'}$ given by $\gamma_i\mapsto W_i$ where $W_i$'s are words in the generators of $\Gamma_{P'}$.  This homomorphism lifts to a homomorphism of free groups $\widehat{\Phi}:F_r\to F_{r'}$ defined likewise.  Thus, we have a function between manifolds $\widehat{\Phi}_*:G^{r'}\to G^r$ given by $(g_1,...,g_{r'})\mapsto (W_1(g_1,...,g_{r'}),...,W_{r}(g_1,...,g_{r'}))$; $\widehat{\Phi}_*$ is analytic since the group operations in $G$ are analytic.  The same can be said for $\widehat{\Phi^{-1}}_*$.  Since these two analytic maps are bijective on the images of $\iota_P$ and $\iota_{P'}$ respectively, we conclude that $\Phi_*:\hm(\Gamma_Q,G)\to \hm(\Gamma_P,G)$ given by $\Phi_*(\rho)=\rho\circ \Phi$ is an {\it analytic isomorphism}.  In short, the structure of an analytic variety on $\hm(\Gamma, G)$ is independent of presentation (up to natural equivalence).

The group $G$ acts (analytically) by conjugation on $\hm(\Gamma, G)$.  Define $\hm(\Gamma, G)^*$ to be the subspace of homomorphisms with {\it closed} conjugation orbits, and let $\X(\Gamma,G):=\hm(\Gamma,G)^*/G$ be the quotient space by conjugation.  The space $\X(\Gamma, G)$ is known as the {\it $G$-character variety of $\Gamma$}.  

In general, if $\Gamma$ is the fundamental group of a manifold $M$, $\X(\Gamma, G)$ corresponds to a moduli space of flat principal $G$-bundles over $M$ by associating to each bundle its holonomy and considering two such bundles equivalent if their holonomies are topologically indistinguishable.

In the case when $G$ is compact, then $\X(\Gamma, G)$ is simply the quotient space $\hm(\Gamma, G)/G$.  When $G$ is complex reductive, then $\X(\Gamma, G)$ is homeomorphic to the geometric invariant theoretic quotient $\hm(\Gamma, G)/\!\!/G$ with the Euclidean topology (by \cite[Theorem 2.1]{FlLa3}), and homotopic to the non-Hausdorff quotient $\hm(\Gamma, G)/G$ (by \cite[Proposition 3.4]{FLR}).  More generally, when $G$ is real reductive $\X(\Gamma, G)$ embeds into a real affine space as a closed subspace and hence is Hausdorff \cite{RS}.

When $\Gamma$ is the fundamental group of a genus 1 orientable surface with one boundary component (a one-holed torus for short), $\Gamma$ is isomorphic to a rank 2 free group $F_2$.  We now review the structure of the character varieties we concern ourselves with in this case.

\subsection{$\SL(2,\C)$ and $\SU(2)$ Character Variety of $F_2$}

Let $G=\SL(2,\C)$, $K=\SU(2)$, and $F_2=\langle a,b\rangle$ a free group of rank 2.  
Let $w\in F_2$ and $\tr_w:\hm(F_2,G)\to \C$ be the conjugation invariant regular function defined by $\tr_w(\rho)=\tr(\rho(w))$.  Under the identification $\hm(F_2,G)\cong G^2$, letting $A:=\rho(a)$ and $B:=\rho(b)$,  $\tr_w(A,B)=\tr(w(A,B)),$ where $w(A,B)$ is the word $w$ with $a$ replaced by $A$ and $b$ replaced by $B$.

In these terms, the Fricke-Vogt Theorem says that the coordinate ring $\C[\X(F_2,G)]=\C[\tr_a,\tr_b,\tr_{ab}]$ and so $\X(F_2,G)\cong \C^3$.

Now, $\X(F_2,K)$ naturally embeds into $\X(F_2,G)$ by Weyl's unitary trick.  For a representation $\rho:F_2\to G$ to be in $K=\SU(2)$ it is necessary that the functions $\tr_w$ take value in $[-2,2]$ since a matrix in $\SU(2)$ is conjugate to one with eigenvalues $\{e^{i\theta}, e^{-i\theta}\}$ and so has trace equal to $2\cos(\theta)$.  Conversely, one can show that it sufficient (up to conjugation) to demand that $\tr_a, \tr_b, \tr_{ab},$ and $\tr_{aba^{-1}b^{-1}}=\tr_a^2+\tr_b^2+\tr_{ab}^2-\tr_a\tr_b\tr_{ab}-2$ take value in $[-2,2]$; see \cite[Theorem 6.5]{FlLa1}.  The resulting semi-algebraic space is a closed 3-ball in $\R^3$ (Figure \ref{su21holed}).

\begin{figure}[ht!]
\includegraphics[scale=0.4]{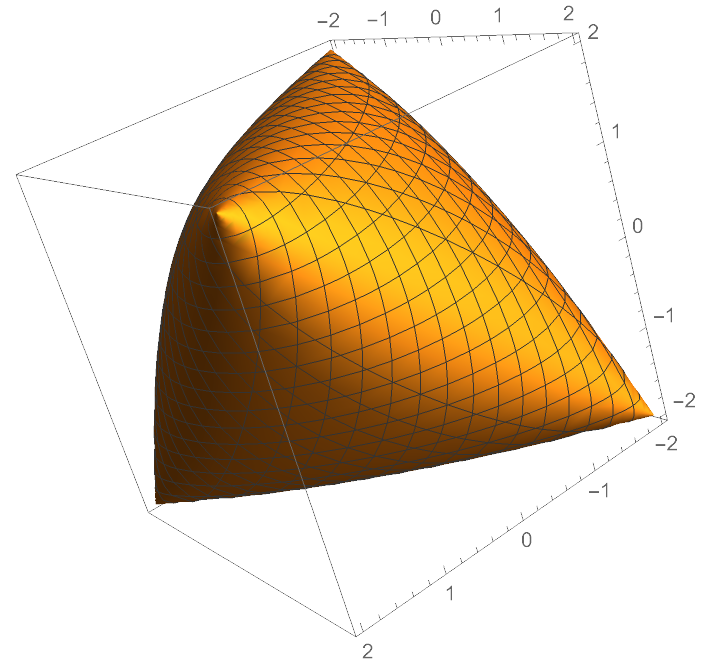} \caption{$\X(F_2,\SU(2))$}
\label{su21holed} 
\end{figure}

\subsection{$\SL(3,\C)$ and $\SU(3)$ Character Variety of $F_2$}
The comparable problem for $G=\SL(3,\C)$ and $K=\SU(3)$ was solved in \cite{La0,La1,FlLa1}, but is much harder.  Here follows a summary.  In this case, there are more generators necessary.  In particular, the coordinate ring $\C[\X(F_2,G)]$ is (minimally) generated by the nine functions $$\{\tr_a,\tr_{a^{-1}},\tr_b,\tr_{b^{-1}},\tr_{ab},\tr_{a^{-1}b^{-1}},\tr_{ab^{-1}},\tr_{a^{-1}b},\tr_{aba^{-1}b^{-1}}\}.$$  These generators satisfy a single relation $\tr_{aba^{-1}b^{-1}}^2-P\tr_{aba^{-1}b^{-1}}+Q$ where $P,Q\in \C[\tr_a,\tr_{a^{-1}},\tr_b,\tr_{b^{-1}},\tr_{ab},\tr_{a^{-1}b^{-1}},\tr_{ab^{-1}},\tr_{a^{-1}b}]$.  Consequently, $\X(F_2,G)$ is a hypersurface in $\C^9$ that branched double-covers $\C^8$.

If we set $x=\tr_a$, $y=\tr_b$, $z=\tr_{ab}$, $t=\tr_{ab^{-1}}$,$X=\tr_{a^{-1}}$, $Y=\tr_{b^{-1}}$, $Z=\tr_{a^{-1}b^{-1}}$, and $T=\tr_{a^{-1}b}$, then:
$$P=t T-t X y-T x Y+x X y Y+x X-x y Z-X Y z+y Y+z Z-3$$ and
\begin{eqnarray*}Q&=&-2 t^2 x Y+t^2 X Z+t^2 y z+t^3+t T x X+t T y Y+t T z Z-6 t T\\ &+& t x^2y+t x^2 Y^2-t x X^2 y-t x X Y Z-t x y Y z-3 t x z+t x Z^2\\&+&t X^2z-t X y^2 Y+3 t X y+t X Y^2+t y^2 Z+t Y z^2-3 t Y Z+T^2 x z\\&-&2 T^2X y+T^2 Y Z+T^3-T x^2 X Y+T x^2 Z-T x X y z+T x y^2\\&-&T x y Y^2+3 Tx Y+T X^2 y^2+T X^2 Y-T X y Y Z+T X z^2\\&-&3 T X Z-3 T y z+T y Z^2+TY^2 z+x^2 X^2 y Y-x^2 X y Z+x^2 y^2 z\\&-&x^3 y Y+x^2 Y z+x^3-x X^2 Yz+x X y^2 Y^2-x X y^3+x X y Y\\&-&x X Y^3+x X z Z-6 x X-x y^2 Y Z-2 xy z^2+3 x y Z+x Y^2 Z\\&-&X^3 y Y+X^2 y Z+X^2 Y^2 Z+X^3+X y^2 z-X yY^2 z+3 X Y z\\&-&2 X Y Z^2+y^3+y Y z Z-6 y Y+Y^3+z^3
-6 z Z+Z^3+9.\end{eqnarray*}

By \cite{FlLa1}, $\X(F_2,G)$ is homotopic to $\X(F_2,K)$ and $\X(F_2,K)$ is homeomorphic to an 8-sphere $S^8$.
For representation taking values in $\SU(n)$, since $A^{-1}=\overline{A}^\dagger$ in $\SU(n)$, we have that $\tr_{w^{-1}}=\overline{\tr_w}$ for any word $w\in F_2$.  Thus, the real coordinate ring of $\X(F_2,K)$ is generated by the real and imaginary parts of $$\{\tr_a,\tr_b,\tr_{ab},\tr_{ab^{-1}},\tr_{aba^{-1}b^{-1}}\},$$ subject to the relations $\mathrm{Re}(\tr_{aba^{-1}b^{-1}})=P/2$ and $$\mathrm{Im}(\tr_{aba^{-1}b^{-1}})=\pm\sqrt{Q-P^2/4}.$$  Note that both $P,Q$ will be real polynomials as $P$ is the sum $$P=\tr_{aba^{-1}b^{-1}}+\tr_{bab^{-1}a^{-1}}=2\mathrm{Re}(\tr_{aba^{-1}b^{-1}})$$ while $Q$ is the product $$Q=\tr_{aba^{-1}b^{-1}}\cdot\tr_{bab^{-1}a^{-1}}=|\tr_{aba^{-1}b^{-1}}|^2,$$ and hence will simplify to {\it real} polynomials in the real and imaginary parts of $\{\tr_a,\tr_b,\tr_{ab},\tr_{ab^{-1}}\}$.
 
\section{Relative Character varieties}\label{sec-3}

Let $\Sigma_{n,g}$ be a compact connected orientable surface of genus $g\,\geq\, 0$ with $n\,\geq\,0$
boundary components; we assume that $n\,\geq\, 2$ if $g\,=\,0$ since otherwise the surface is simply-connected
and the moduli spaces we will consider are trivial. Pick a base point $*$ in the interior of $\Sigma_{n,g}$.
The fundamental group of $\Sigma_{n,g}$ admits the presentation:
$$\pi_1(\Sigma_{n,g},\,*)\,\cong\, \left\langle a_1,\,b_1,\,\dots,\,a_g,\,b_g,\,c_1,\,\dots,\,c_n\ \Big\vert\ 
\prod_{i=1}^g[a_i,b_i]\prod_{j=1}^nc_j\,=\,1\right\rangle,$$ where $[x,\,y]\,=\,xyx^{-1}y^{-1}$ is the commutator.

For $G$ compact or complex reductive, let 
$\X_{n,g}(G):=\X(\pi_1(\Sigma_{n,g}),G).$  The dimension (complex if $G$ is complex reductive and real if $G$ is compact) of $\X_{n,g}(G)$ is $$-\chi(\Sigma_{n,g})\dim G +\zeta_{n,g},$$ where $\zeta_{n,g}$ depends on the rank of $G$, $g,$ and $n$ (see \cite[Lemma 2.2]{BHJL} for details).

When $n\,>\,0$, for every $1\, \leq\, i\, \leq\, n$ define the boundary map
$$\mathfrak{b}_i\,:\,\X_{n,g}(G)\,\longrightarrow\, \X(\Z,G)$$ by sending a representation class $[\rho]$ to $[\rho_{|_{c_i}}]$. Subsequently, we define $$\mathfrak{b}_{n,g}\,:=\,(\mathfrak{b}_1,\,\dots,\,\mathfrak{b}_n)
\,:\,\X_{n,g}(G)\,\longrightarrow\, (\X(\Z,G))^{ n}.$$ 

Let $\tau\,\in\, \mathfrak{b}_{n,g}\left(\X_{n,g}(G)\right)\,\subset\, (\X(\Z,G))^n$ be a point in the image of the 
boundary map and define $\mathfrak{L}_\tau\,:=\,\mathfrak{b}_{n,g}^{-1}(\tau).$ The singular locus of $\X_{n,g}(G)$ is 
a proper closed sub-variety; denote its complement by $\mathcal{X}_{n,g}(G)$. So $\mathcal{X}_{n,g}(G)$ is a
manifold that is dense in $\X_{n,g}(G)$. Since $\mathfrak{b}_{n,g}$ is dominant, its regular values are generic.
Thus, at such a point, $\mathcal{L}_\tau\,:=\,\mathfrak{L}_\tau\cap\mathcal{X}_{n,g}(G)$ is a submanifold of dimension 
$$d_\tau:=\chi(\Sigma_{n,g})\dim G +\zeta_{n,g}-n(\mathrm{Rank}(G)).$$

It is shown in \cite{La2} that $\cup_\tau\mathcal{L}_\tau$ foliate $\mathcal{X}_{n,g}(G)$ by symplectic 
submanifolds, making $\mathcal{X}_{n,g}(G)$ a Poisson manifold (real if $G$ is compact and complex if $G$ is reductive). This structure continuously extends over 
all of $\X_{n,g}(G)$ making it a Poisson variety; a variety whose sheaf of regular functions is a sheaf of Poisson 
algebras (see \cite{BLR} for details).

\subsection{Symplectic Measure}
On each symplectic leaf $\mathcal{L}_\tau$ there is a symplectic volume by considering the $(d_\tau/2)$-th exterior power of the 2-form $\omega$.  

The group of automorphisms $\mathrm{Aut}(\Gamma)$ acts on $\hm(\Gamma, G)$ by $\alpha\cdot \rho:=\rho\circ \alpha^{-1}$ for any groups $\Gamma$ and $G$.  This action descends to an action of $\mathrm{Out}(\Gamma):=\mathrm{Aut}(\Gamma)/\mathrm{Inn}(\Gamma)$ on the $G$-conjugation quotient $\hm(\Gamma, G)/G$ by $[\alpha]\cdot[\rho]=[\rho\circ \alpha^{-1}]$.  By continuity, this restricts to an action of  $\mathrm{Out}(\Gamma)$ on $\X(\Gamma, G)$ when $G$ is real reductive.  

When $\Gamma=\pi_1(\Sigma_{n,g})$, we can define $\mathrm{Out}_{n,g}(\Gamma)\leq \mathrm{Out}(\Gamma)$ to be the subgroup that fixes the boundary curves (up to homotopy) in $\Sigma_{n,g}$.  Then 
$\mathrm{Out}_{n,g}(\Gamma)$ is isomorphic to the mapping class group of $\Sigma_{n,g}$, that is, $$\mathrm{Out}_{n,g}(\Gamma)\cong\mathrm{MCG}(\Sigma_{n,g}).$$  Since the mapping class group preserves the boundary curves and the underlying surface, we obtain an action on $\mathcal{L}_\tau$.  
In \cite{Hueb}, it is shown that when $G$ is compact, this invariant measure is finite and so can be taken to be a probability measure.  Morally this follows since the symplectic form, and so the symplectic volume (in the same measure class as Lebesgue measure), continuously extends over $\mathfrak{L}_\tau$, which is compact, and so the induced measure is finite on $\mathfrak{L}_\tau$.  Since the singular locus is a proper subvariety it has measure zero.  Hence the restricted measure on the smooth locus $\mathcal{L}_\tau$ stays finite.

\subsection{Relative Character Varieties of $\Sigma_{1,1}$} Throughout this section we let $\Gamma=\pi_1(\Sigma_{1,1})\cong F_2=\langle a,b\rangle$, and we let $c^{-1}=aba^{-1}b^{-1}$ be the boundary homotopy class in $\Sigma_{1,1}$.

Recall that for $G=\SU(2)$, the character variety $\X(\Gamma, G)$ is homeomorphic to closed ball in $\R^3$  (Figure \ref{su21holed}).  Let $$\kappa:=\tr_{c}=x^2+y^2+z^2-xyz-2.$$ The ball is parametrized by $x=\tr_a,$ $y=\tr_b,$ and $z= \tr_{ab}$ subject to the conditions: $-2\leq x,y,z,\kappa\leq 2$.  The relative character varieties in this case are determined by fixing the value of $\mathfrak{b}_{1,1}=\kappa$.  So each symplectic leaf is a 2-sphere as $\kappa$ is a height function for the ball\footnote{See \url{https://youtu.be/lRdsxwr_dHI} for a video of the symplectic foliation.}.  The value $\kappa=-2$ is a point (degenerate 2-sphere) and the other 2-spheres are all smooth except the level $\kappa=2$ which is topologically still a 2-sphere  (the boundary of Figure \ref{su21holed}), but has 4 orbifold singularities at the four central homomorphisms $(\pm I_2,\pm I_2)$, each locally homeomorphic to a cone over $\R\mathrm{P}^1$.

\begin{figure}[ht!]
\includegraphics[scale=0.4]{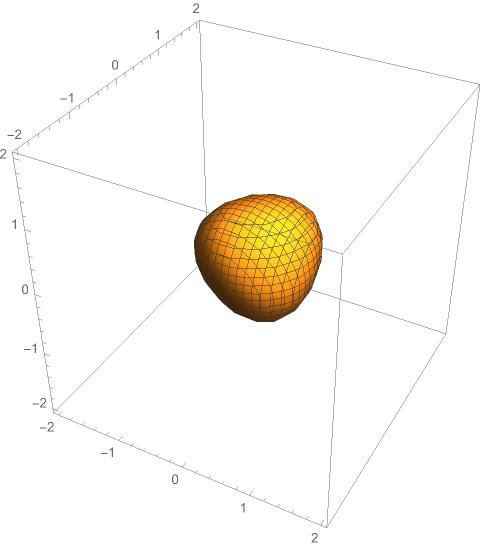} \caption{$\kappa=-1.3$}
\label{su21holedsmooth} 
\end{figure}

When $G=\SU(3)$, the character variety $\X(\Gamma,G)$ is homeomorphic to an 8-sphere (topologically, but it is singular in the trace coordinates) by \cite{FlLa1}.  To understand the relative character varieties in this case we need to study the boundary map.  There is one boundary $c$ but conjugacy classes are determined by two functions: $\tr_c$ and $\tr_{c^{-1}}$. Since $\tr_c=\overline{\tr_{c^-1}}$ we need only consider the real and imaginary parts of $\tr_{c^{-1}}$ to determine the boundary conjugacy class.  From earlier, we know  $\mathrm{Re}(\tr_{c^{-1}})=P/2$ and $\mathrm{Im}(\tr_{c^{-1}})=\pm\sqrt{Q-P^2/4}$ where $P,Q$ are the polynomials defining $\X(F_2,\SL(3,\C))$ which are real valued on $\X(F_2,\SU(3))$.  So $$\mathfrak{b}_{1,1}=(P/2,\pm\sqrt{Q-P^2/4})$$ is the map $\X(\Gamma,G)\to \X(\Z,\SU(3))$ where $\Delta:=\X(\Z,\SU(3))$ is depicted in Figure \ref{triangle}.  

\begin{figure}[ht!]
\includegraphics[scale=0.4]{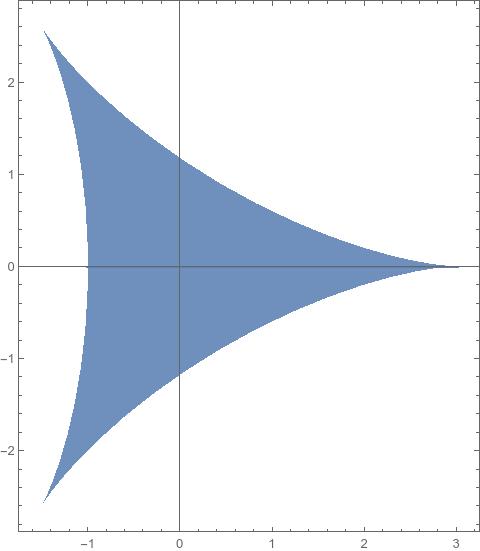} \caption{The \it{deltoid} $\Delta:=\X(\Z,\SU(3))$.}
\label{triangle} 
\end{figure}

The fibres of $\mathfrak{b}_{1,1}$ are the relative character varieties discussed earlier; denote $\mathfrak{L}_{(\zeta, \eta)}:=\mathfrak{b}_{1,1}^{-1}(\zeta, \eta)$ for any $(\zeta,\eta)\in \X(\Z,\SU(3))$.  Note that $\Delta\cap \R=[-1,3].$ Let $\Delta^\circ$ and $\partial\Delta$ be the interior and boundary of $\Delta$, respectively.

 We now consider the possibilities for these fibres.

\begin{proposition} Let $(\zeta,\eta)\in \Delta$.  Then:
\begin{enumerate}
 \item $\mathfrak{L}_{(\zeta, \eta)}$ admits the structure of a normal projective variety over $\C$.
 \item $\mathfrak{L}_{(\zeta, \eta)}\cong \mathfrak{L}_{(\zeta, -\eta)}$.
 \item $\mathfrak{L}_{(3, 0)}\cong \X(\Z^2,\SU(3))\cong\C\mathrm{P}^2$.
 \item $\mathfrak{L}_{(-3/2, 3\sqrt{3}/2)}\cong \mathfrak{L}_{(-3/2, -3\sqrt{3}/2)}\cong \{*\},$
 \item  $\mathfrak{L}_{(\zeta, \eta)}$ is a smooth 6-manifold if $(\zeta,\eta)\in\Delta^\circ-\Delta\cap\R.$
 \item If $(\zeta,\eta)\in\partial \Delta-(\Delta\cap\R\cup\{(-3/2, \pm 3\sqrt{3}/2)\})$, then $\mathfrak{L}_{(\zeta, \eta)}$ is a smooth 4-manifold.
 \item $\mathfrak{L}_{(\zeta, 0)}$ is a singular 6 dimensional variety with 4 dimensional singular locus if $(\zeta,0)\in\Delta^\circ\cap\R.$  The singular locus consists of $\mathsf{S}(\mathsf{U}(1) \times \mathsf{U}(2))$-valued representations. 
 \item $\mathfrak{L}_{(-1, 0)}$ is a singular 4 dimensional variety with a 2 dimensional singular locus consisting of reducible representations whose commutator is $\mathrm{diag}(-1,-1,1)$.
 
\end{enumerate}

\end{proposition}

\begin{proof}

Giving a complex structure on $\Sigma_{1,1}$, viewing the boundary as a puncture (removing a point), produces the structure of a normal projective variety on $\mathfrak{L}_{(\zeta, \eta)}:=\mathfrak{b}_{1,1}^{-1}(\zeta, \eta)$ by identifying it (up to homeomorphism) with the moduli space of semi-stable principal bundles on $\Sigma_{1,1}$ as in \cite{B-R}.  Under this identification, irreducible representations correspond to the Zariski open subset of stable bundles.  In all cases, the smooth locus is exactly the irreducible representations.
 
Next, for any $(\zeta,\eta)\in\Delta$, we have $\mathfrak{b}_{1,1}^{-1}(\zeta, \eta)\cong\mathfrak{b}_{1,1}^{-1}(\zeta, -\eta)$ since $\tr_{c^{-1}}^2-P\tr_{c^{-1}}+Q=0$ has at most two solutions $\{\zeta\pm i\eta\}$.  So we only need to analyze fibres $\mathfrak{b}_{1,1}^{-1}(\zeta, \eta)$ with $\eta\geq 0$.

To every pair $(\zeta, \eta)\in \Delta$, there exists a unique conjugation class in $\SU(3)$ and such a conjugation class is represented by a diagonal matrix.  We refer to this diagonal matrix (unique up to permutation of its eigenvalues) as the {\it boundary matrix}.

If the boundary matrix is central (corresponding to the three vertices in Figure \ref{triangle}), there are two possibilities.  By \cite{ACG}, $$\mathfrak{L}_{(3, 0)}\cong \X(\Z^2,\SU(3))\cong\C\mathrm{P}^2,$$ otherwise we have $$\mathfrak{L}_{(-3/2, 3\sqrt{3}/2)}\cong \mathfrak{L}_{(-3/2, -3\sqrt{3}/2)}\cong \{*\}$$ is a point by \cite[Proposition 7.1]{GLX}.

Another way to realize $\mathfrak{L}_{(\zeta, \eta)}$ is to first pick a boundary matrix $C$, that is, a matrix so $\tr(C)=\zeta+i\eta$.  Then: $$\mathfrak{L}_{(\zeta, \eta)}\cong \{(A,B)\in \SU(3)^2\ |\ [A,B]=C\}/\{S\in \SU(3)\ |\ SCS^{-1}=C\},$$ where the stabilizer $\{S\in \SU(3)\ |\ SCS^{-1}=C\}$ acts by diagonal conjugation.

The next case to consider is the opposite extreme: the boundary matrix has three distinct eigenvalues.  This case corresponds to fibres over $\Delta^\circ-\Delta\cap\R.$  Each such fibre $\mathfrak{L}_{(\zeta, \eta)}$ consists of conjugacy classes of irreducible representations and so is smooth since the conjugation action of $\mathsf{P}\SU(3)$ is free on irreducible representations (by Shur's Lemma).  The dimension is 6 since $\X(F_2,\SU(3))$ is dimension 8 and we are imposing two real conditions by fixing the boundary.  

If $(\zeta,\eta)\in\partial \Delta-(\Delta\cap\R\cup\{(-3/2, \pm 3\sqrt{3}/2)\})$, then the boundary matrix is $\mathrm{diag}(a,a,1/a^2)$ with $a^2\not=1\not=a^3$.  As with the previous case, we have only irreducible representations and so $\mathfrak{L}_{(\zeta, \eta)}$ is again smooth.  But with a repeated eigenvalue, the stabilizer of the boundary matrix is dimension 4 (as opposed to dimension 2 in the previous case).  Consequently, $\mathfrak{L}_{(\zeta, \eta)}$ is a smooth 4-manifold.

$(\zeta,\eta)\in\Delta\cap \R$ if and only if the boundary matrix has 1 as an eigenvalue.  There are two cases not yet addressed: $\mathrm{diag}(1,-1,-1)$ and $\mathrm{diag}(1,a,1/a)$ with $a\not= -1$ (the case with $a=1$ is the identity matrix, already addressed above).

In both cases there are singularities (reducible representations) since 1 is an eigenvalue.  In the first case, where the boundary matrix has an eigenvalue that is not a root of unity, the singular locus is exactly $\mathsf{S}(\mathsf{U}(1) \times \mathsf{U}(2))$-valued representations.  The stabilizer of the boundary matrix has dimension 2.  Consequently, $\mathfrak{L}_{(\zeta, 0)}$  has dimension 6 with a 4 dimensional singular locus.

Lastly, $\mathfrak{L}_{(-1, 0)}$ has dimension 4 with a 2 dimensional singular locus as the stabilizer of the boundary matrix $\mathrm{diag}(1,-1,-1)$ has dimension 4.

\end{proof}

\begin{remark}
Given that $\X(F_2,\SU(3))$ is an 8-sphere, it is natural to think that for $(\zeta,\eta)\in\Delta^\circ-\Delta\cap\R,$ the space $\mathfrak{L}_{(\zeta, \eta)}$ is a 6-sphere.  However this is not the case since $\mathfrak{L}_{(\zeta, \eta)}$ admits the structure of a complete variety over $\C$ (of complex dimension 3) which implies its second cohomology group is non-trivial (whereas it would be trivial if it was a 6-sphere).  This remark also shows that $\X(F_2,\SU(3))$ does not admit the structure of a complex complete variety (as it is an 8-sphere). 
\end{remark}

\begin{remark}
The Poisson structure for $\X(F_2,\SU(3))$ can be obtained by restricting the Poisson structure on $\X(F_2,\SL(3,\C))$ obtained in \cite{La2} to $\X(F_2,\SU(3))$.  We have carried this process out in a {\it Mathematica} notebook.  From this we can calculate the bivector and the symplectic form on the smooth part of $\mathfrak{L}_{(\zeta, \eta)}$ for any boundary matrix.  Unlike the case of $\SU(2)$, the formula for the symplectic form is too complicated to write here, although the {\it Mathematica} notebook with the code and formulae are on GitHub \footnote{\url{https://github.com/seanlawton/Non-ergodicity-on-SU-2-and-SU-3-character-varieties-of-the-once-punctured-torus}}. 
\end{remark}

\section{Brown's theorem revisited}\label{sec-4}

Recall that the fundamental group of $\Sigma_{1,1}$ is a free group $\langle a,b\rangle$.  Let $\alpha,\beta$ represent curves in $\Sigma_{1,1}$ whose homotopy classes are $a,b$ respectively.  The Dehn twists $$\tau_{\alpha}=\left(\begin{array}{cc}1& 1\\ 0&1\end{array}\right) \textrm{ and } \tau_{\beta} = \left(\begin{array}{cc}1& 0\\ 1&1\end{array}\right)$$ about $\alpha$ and $\beta$  act via $$\tau_{\alpha}(\alpha)=\alpha, \quad \tau_{\alpha}(\beta)=\beta\alpha, \quad \tau_{\beta}(\alpha) = \alpha\beta, \quad \tau_{\beta}(\beta)=\beta.$$

In terms of the trace coordinates $x=\textrm{tr}(\rho(a))$, $y=\textrm{tr}(\rho(b))$ and $z=\textrm{tr}(\rho(ab))$ on the $\SU(2)$-character variety 
$$\X(F_2,\SU(2)) \cong \{(x,y,z)\in[-2,2]^3\ |\ -2\leq \kappa(x,y,z)\leq 2\},$$
where $\kappa(x,y,z)=x^2+y^2+z^2-xyz-2$.  The actions of $\tau_{\alpha}$ and $\tau_{\beta}$ are described by: 
\begin{equation}\label{e.3}
\tau_{\alpha}(x,y,z)=(x,z,xz-y) \textrm{ and } \tau_{\beta}^{-1}(x,y,z)=(xy-z,y,x).
\end{equation}
 
\subsection{Brown's theorem} The following statement is the main theorem of Brown \cite{Brown98}: 

\begin{theorem}\label{t.A} Let $h$ be a hyperbolic element of $\SL(2,\mathbb{Z})$. Then, there exists $-2<\ell<2$ such that $h$ does not act ergodically on the level $\kappa^{-1}(\ell)$ of the $\SU(2)$-character variety of the once-punctured torus. 
\end{theorem}

Roughly speaking, Brown establishes Theorem \ref{t.A} along the following lines. One starts by performing a blow up at the origin $\kappa^{-1}(-2)=\{(0,0,0)\}$ in order to think of the action of $h$ on $\X_{1,1}(\SU(2))$ as a one-parameter family $h_{(\ell)}$, $-2\leq \ell \leq 2$, of area-preserving maps of the $2$-sphere such that $h_{(-2)}$ is a finite order element of $\SO(3)$. In this way, we have that $h_{(\ell)}$ is a non-trivial one-parameter family going from completely elliptic behaviour at $\ell=-2$ to non-uniformly hyperbolic behaviour at $\ell=2$. This scenario suggests that the conclusion of Theorem \ref{t.A} can be derived via KAM theory in the elliptic regime. In the sequel, we revisit Brown's ideas leading to Theorem \ref{t.A} (with an special emphasis on its KAM theoretical aspects). 

\subsection{Blow up of the origin}  The origin $\kappa^{-1}(-2)$ of the character variety $\X_{1,1}(\SU(2))$ can be blown up into a sphere of directions $S_{-2}$. The action of $\SL(2,\mathbb{Z})$ on $S_{-2}$ factors through an octahedral subgroup of $\SO(3)$: this follows from the fact that \eqref{e.3} implies that the generators $\tau_{\alpha}$ and $\tau_{\beta}$ of $\SL(2,\mathbb{Z})$ act on $S_{-2}$ as
$$\tau_{\alpha}|_{S_{-2}}(\dot{x},\dot{y},\dot{z}) = (\dot{x},\dot{z},-\dot{y}), \quad \tau_{\beta}^{-1}|_{S_{-2}}(\dot{x},\dot{y},\dot{z})=(-\dot{z},\dot{y},\dot{x}).$$ 
In this way, each element $h\in \SL(2,\mathbb{Z})$ is related to a root of unity 
$$\lambda_{-2}(h)\in \U(1)=\{w\in \mathbb{C}\ |\ |w|=1\}$$ 
of order $\leq 4$ coming from the eigenvalues of the derivative of $h|_{S_{-2}}$ at any of its fixed points.  

\begin{example} The hyperbolic element $\left(\begin{array}{cc} 2 & 1 \\ 1 & 1\end{array}\right) = \tau_{\alpha}\tau_{\beta}$ acts on $S_{-2}$ via the element $(\dot{x},\dot{y},\dot{z})\mapsto (\dot{z},-\dot{x},-\dot{y})$ of $\SO(3)$ of order $3$. 
\end{example} 

\subsection{Bifurcations of fixed points} A hyperbolic element $h\in \SL(2,\mathbb{Z})$ induces a non-trivial polynomial automorphism of $\mathbb{R}^3$ whose restriction to $\kappa^{-1}([-2,2])$ describe the action of $h$ on $\X(F_2,\SU(2))$. In particular, the set $L_h$ of fixed points of this polynomial automorphism in $\kappa^{-1}([-2,2])$ is a semi-algebraic set of dimension less than $ 3$. 

One can exploit the fact that $h$ acts on the level sets $\kappa^{-1}(\ell)$, $\ell\in[-2,2]$, via area-preserving maps, to compute the Zariski tangent space to $L_{h}$ in order to verify that $L_h$ is one-dimensional (\cite[Proposition 5.1]{Brown98}). 

Moreover, this calculation of the Zariski tangent space can be combined with the fact that any hyperbolic element $h\in \SL(2,\mathbb{Z})$ has a discrete set of fixed points in $\mathbb{R}^2/\mathbb{Z}^2$ and, a fortiori, in $\kappa^{-1}(2)=\X(\Z^2,\SU(2))$ to get that $L_h$ is transverse to $\kappa$ except at its discrete subset of singular points and, hence, $L_h\cap \kappa^{-1}(\ell)$ is discrete for all $-2\leq \ell\leq 2$ (\cite[Proposition 5.2]{Brown98}). 

\begin{example}  The hyperbolic element $\tau_{\alpha}\tau_{\beta}=\left(\begin{array}{cc} 2 & 1 \\ 1 & 1\end{array}\right)$ acts on the character variety $\X(F_2,\SU(2))$ via the polynomial automorphism 
$$(x,y,z)\mapsto (z, zy-x, z(zy-x)-y)$$ by Equation \eqref{e.3}. Thus, the corresponding set of fixed points is given by the equations  $$x=z, \quad y=zy-x, \quad z=z(zy-x)-y$$  describing an embedded curve in $\mathbb{R}^3$. 
\end{example} 

In general, the eigenvalues $\{\lambda(p), \lambda(p)^{-1}\}$ of the derivative at $p\in L_h$ of the action of a hyperbolic element $h\in \SL(2, \mathbb{Z})$ on $\kappa^{-1}(\kappa(p))$ can be continuously followed along any irreducible component $\mathcal{C}_h\ni p$ of $L_h$. 

Furthermore, it is not hard to check that $\lambda$ is not constant on $\mathcal{C}_h$ (\cite[Lemma 5.3]{Brown98}). Indeed, this happens because there are only two cases: the first possibility is that $\mathcal{C}_h$ connects $\kappa^{-1}(-2)$ and $\kappa^{-1}(2)$ so that $\lambda$ varies from $\lambda_{-2}(h)\in \U(1)$ to the unstable eigenvalue of $h$ acting on $\mathbb{R}^2/\mathbb{Z}^2$; the second possibility is that $\mathcal{C}_h$ becomes tangent to $\kappa^{-1}(\ell)$ for some $-2<\ell<2$ so that the Zariski tangent space computation mentioned above reveals that $\lambda$ varies from $1$ (at $\mathcal{C}_h\cap\kappa^{-1}(\ell)$) to some value $\neq 1$ (at any point of transverse intersection between $\mathcal{C}_h$ and a level set of $\kappa$). 

\subsection{Detecting Brjuno elliptic periodic points}

The discussion of the previous two subsections allows us to show that some aspects of the action of a hyperbolic element $h\in \SL(2,\mathbb{Z})$ fit the assumptions of KAM theory. Before entering into this matter, recall that $e^{2\pi i\theta}\in \U(1)$ is {\it Brjuno} whenever $\theta$ is an irrational number whose continued fraction has partial convergents $(p_k/q_k)_{k\in\mathbb{Z}}$ satisfying: 
$$\sum\limits_{k=1}^{\infty} \frac{\log q_{k+1}}{q_k}<\infty.$$ 
For our purposes, it is important to note that the Brjuno condition has full Lebesgue measure on $\U(1)$. Indeed, the set of Diophantine irrationals is a proper subset of the set of Brjuno irrationals
and it is well known that already the set of Roth Diophantine irrationals has full measure.  We recall that by definition $\theta$ is a Diophantine irrational of exponent $\gamma>0$ if  $q_{k+1} = O(q_k^{1+\gamma})$ 
for all $k \in N$  and  it is a Roth irrational if it is Diophantine of exponent $\gamma$ for all $\gamma >0$.

Let $h\in \SL(2,\mathbb{Z})$ be a hyperbolic element. We have three possibilities for the limiting eigenvalue $\lambda_{-2}(h)\in \U(1)$: it is not real, it equals $1$ or it equals $-1$.

If the limiting eigenvalue $\lambda_{-2}(h)\in \U(1)$ is not real, then we take an irreducible component $\mathcal{C}_h$ intersecting the origin $\kappa^{-1}(-2)$. Since $\lambda$ is not constant on $\mathcal{C}_h$ we conclude that $\lambda(\mathcal{C}_h)$ contains an open subset of $\U(1)$. Thus, we can find some $-2<\ell<2$ such that $\{p\}=\mathcal{C}_h\cap \kappa^{-1}(\ell)$ has a Brjuno eigenvalue $\lambda(p)$, i.e., the action of $h$ on $\kappa^{-1}(\ell)$ has a Brjuno fixed point. 

If the limiting eigenvalue is $\lambda_{-2}(h)=1$, we use the Lefschetz fixed point theorem on the sphere $\kappa^{-1}(\ell)$ with $\ell$ close to $-2$ to locate an irreducible component $\mathcal{C}_h$ of $L_h$ such that $\{p_\ell\}=\mathcal{C}_h\cap\kappa^{-1}(\ell)$ is a fixed point of positive index of $h|_{\kappa^{-1}(\ell)}$ for $\ell$ close to $-2$. On the other hand, it is known that an isolated fixed point of an orientation-preserving surface homeomorphism which preserves area has index $<2$. Therefore, $p_\ell$ is a fixed point of $h|_{\kappa^{-1}(\ell)}$ of index $1$ with multipliers $\lambda(p_\ell), \lambda(p_k)^{-1}$ close to $1$ whenever $k$ is close to $-2$. Since a hyperbolic fixed point with positive multipliers has index $-1$, it follows that $p_\ell$ is a fixed point with $\lambda(p_\ell)\in \U(1)\setminus\{1\}$ when $\ell$ is close to $-2$. In particular, $\lambda(\mathcal{C}_h)$ contains an open subset of $\U(1)$ and, hence, we can find some $-2<\ell<2$ such that $p_\ell$ has a Brjuno multiplier $\lambda(p_\ell)$. 

If the limiting eigenvalue is $\lambda_{-2}(h)=-1$, then $h^2$ is a hyperbolic element with limiting eigenvalue $\lambda_{-2}(h^2)=1$. From the previous paragraph, it follows that we can find some $-2<\ell<2$ such that $\kappa^{-1}(\ell)$ contains a Brjuno elliptic fixed point of $h^2|_{\kappa^{-1}(\ell)}$. 

In any event, the arguments above give the following result (\cite[Theorem 4.4]{Brown98}): 

\begin{theorem}\label{t.Bruno} Let $h\in \SL(2,\mathbb{Z})$ be a hyperbolic element. Then, there exists $-2<\ell<2$ such that $h|_{\kappa^{-1}(\ell)}$ has a periodic point of period one or two with a Brjuno multiplier. 
\end{theorem} 

\begin{remark} 
It is natural to ask if one can avoid the above non-constructive existence argument by explicitly computing a Brjuno fixed point given our concrete setting.  However, we expect (after unsuccessful attempts) that such a computation is rather difficult and lengthy. One reason to expect this is that a way of finding a Brjuno multiplier is to look for algebraic multipliers (by Roth's Theorem \cite{Roth}), but deciding whether analytic functions take algebraic values is, generally speaking, a difficult problem.
\end{remark}

\subsection{Moser's twist theorem \& R\"ussmann's stability theorem} At this point, the idea to derive Theorem \ref{t.A} is to combine Theorem \ref{t.Bruno} with KAM theory ensuring the stability of certain types of elliptic periodic points.  

Recall that a periodic point is called \emph{stable} whenever there are arbitrarily small neighborhoods of its orbit which are invariant. In particular, the presence of a stable periodic point implies the non-ergodicity of an area-preserving map. 

A famous stability criterion for fixed points of area-preserving maps is Moser's twist theorem \cite{Moser73}. This result can be stated as follows. Suppose that $f$ is an area-preserving $C^r$-map, $r\geq 4$, having an elliptic fixed point at the origin $(0,0)\in \mathbb{R}^2$ with multipliers $\{e^{2\pi i\theta}$, $e^{-2\pi i\theta}\}$ such that $n\theta\notin\mathbb{Z}$ for $n=1, 2, 3, \dots, r$. After performing an appropriate area-preserving change of variables (tangent to the identity at the origin), one can put $f$ into its \emph{Birkhoff normal form}, i.e., $f$ has the form: 
$$\left(\begin{array}{c}\xi \\ \eta\end{array}\right)\!\! \mapsto\!\! \left(\begin{array}{c} \xi\cos(\sum\limits_{n=0}^s\gamma_n(\xi^2+\eta^2)^n)-\eta\sin(\sum\limits_{n=0}^s \gamma_n(\xi^2+\eta^2)^n) \\ 
\xi\sin(\sum\limits_{n=0}^s \gamma_n(\xi^2+\eta^2)^n)+\eta\cos(\sum\limits_{n=0}^s \gamma_n(\xi^2+\eta^2)^n)\end{array}\right) + h(\xi,\eta)$$ 
where $s=[r/2]-1$, $\gamma_0=2\pi\theta$, $\gamma_1, \dots, \gamma_s$ are uniquely determined Birkhoff invariants and $h(\xi,\eta)$ denotes higher order terms.   In Appendices \ref{appa} and \ref{appb} we will discuss the Birkhoff normal form and its relation to KAM theory in more detail.

\begin{theorem}[Moser twist theorem] Let $f$ be an area-preserving map as in the previous paragraph. If $\gamma_n\neq0$ for some $1\leq n\leq s$, then the origin $(0,0)\in\mathbb{R}^2$ is a stable fixed point. 
\end{theorem}

The nomenclature ``twist map'' comes from the condition that the ``vertical'' lines $\theta= \theta_0$, for constant $\theta_0$, in an annulus
$(0, r_0] \times S^1$ are sent by the map to lines that ``twist'' always in the positive (or negative) direction around the annulus.  When $f$  has the form $f(r, \theta) = (r, \theta+\mu(r))$ in polar coordinates, where $\mu$ is a smooth function, then $f$ is a twist map on $(0, r_0] \times S^1$ if and only if $\vert \mu'(r) \vert \not =0$ for all $r\in (0, r_0]$. In particular, if $\gamma_n\not=0$ for some $n>0$ (called a {\it twist condition}), then there exists $r_0>0$ such that $f$ is a twist map in the annulus $(0, r_0] \times S^1$  (see for instance  \S 4  in  \cite{MaFo}).

\begin{example} The Dehn twist $\tau_{\alpha}$ induces the polynomial automorphism $\tau_{\alpha}(x,y,z)= (x,z,xz-y)$ on $\X(F_2, \SU(2))=\kappa^{-1}([-2,2])$. Each level set $\kappa^{-1}(\ell)$, $-2<\ell<2$, is a smooth $2$-sphere which is swept out by the $\tau_{\alpha}$-invariant ellipses  $\mathcal{E}_{\ell, x_0}$ obtained from the intersections between $\kappa^{-1}(\ell)$ and the planes of the form $\{x_0\}\times \mathbb{R}^2$. Goldman \cite{Gold6} observed that, after an appropriate change of coordinates, each $\mathcal{E}_{\ell, x_0}$ becomes a circle where $\tau_{\alpha}$ acts as a rotation by angle $\cos^{-1}(x_0/2)$. In particular, the restriction of $\tau_{\alpha}$ to each level set $\kappa^{-1}(\ell)$ is a twist map near its fixed points $(\pm\sqrt{2+\ell},0,0)$. 
\end{example} 

In his original argument, Brown \cite{Brown98} deduced Theorem \ref{t.A} from (a weaker version of) Theorem \ref{t.Bruno} and Moser's twist theorem. However, Brown employed Moser's theorem with $r=4$ while checking only the conditions on the multipliers of the elliptic fixed point but not the twist condition $\gamma_1\neq 0$.

As it turns out, it is not obvious to check the twist condition in Brown's setting (especially because it is not satisfied at the sphere of directions $S_{-2}$). We note that for the cat map (defined in the next section) we do explicitly check the twist condition in Section \ref{su2kam}.

Fortunately, R\"ussmann \cite{R02} discovered that a Brjuno elliptic fixed point of a real-analytic area-preserving map is always stable (independently of the twist conditions):

\begin{theorem}[R\"ussmann]\label{t.Russmann} Any Brjuno elliptic periodic point of a real-analytic area-preserving map is stable.  
\end{theorem} 

\begin{remark} R\"ussmann obtained the previous result by showing that a real-analytic area-preserving map with a Brjuno elliptic fixed point and vanishing Birkhoff invariants (i.e., $\gamma_n=0$ for all $n\in\mathbb{N}$) is analytically linearizable. Note that the analogue of this statement is false in the $C^{\infty}$ category (as a counterexample is given by $(r,\ \theta)\mapsto (r,\ \theta+\rho+e^{-1/r})$). 
\end{remark} 

Consequently, at this stage, the proof of Theorem \ref{t.A} is complete: it suffices to put together Theorems \ref{t.Bruno} and \ref{t.Russmann}. 

\section{Fixed points of the cat map}\label{sec-5}

Let $\Sigma_{1,1}$ be  a surface of genus $1$ with a boundary component $\gamma$. Recall its fundamental group is free of rank 2: $F_2=\langle a,b\rangle$.

The {\it cat map} is $M=\left(\begin{array}{cc}2& 1\\ 1&1\end{array}\right)=\tau_{\alpha}\tau_{\beta}$, where 
$$\tau_{\alpha}=\left(\begin{array}{cc}1& 1\\ 0&1\end{array}\right) \textrm{ and } \tau_{\beta} = \left(\begin{array}{cc}1& 0\\ 1&1\end{array}\right)$$ are the Dehn twists about the curves $\alpha, \beta$ in $\Sigma_{1,1}$ corresponding to $a,b$ respectively.

Also recall that  
$$\tau_{\alpha}(\alpha)=\alpha, \quad \tau_{\alpha}(\beta)=\beta\alpha, \quad \tau_{\beta}(\alpha) = \alpha\beta, \quad \tau_{\beta}(\beta)=\beta.$$ 

\subsection{Fixed points on $\X(F_2, \SU(2))$}
In trace coordinates $x=\textrm{tr}(\rho(a))$, $y=\textrm{tr}(\rho(b))$ and $z=\textrm{tr}(\rho(ab))$, the $\SU(2)$-character variety of $\Sigma_{1,1}$ is described as:
$$\X(F_2, \SU(2)) \cong \{(x,y,z)\in[-2,2]^3\ |\ -2\leq x^2+y^2+z^2-xyz-2\leq 2\}.$$

Since $\tau_{\alpha}(x,y,z)=(x,z,xz-y)$ and $\tau_{\beta}^{-1}(x,y,z)=(xy-z,y,x)$, we see that 
$$M(x,y,z)=(z, zy-x, z(zy-x)-y).$$

Therefore, the fixed point locus of $M$ on $\X(F_2, \SU(2))$ is described by the equations 
$$x = z, \quad y = zy -x, \quad z = z(zy-x)-y$$ 
in trace coordinates (only the first two equations are needed).

The intersection of the fixed point formulae above with $\X(F_2, \SU(2))$ describes a curve passing through the origin $(0,0,0)$ and the point $(2,2,2)$: this can be seen using {\it Mathematica}. 

The matrix counterparts of these fixed points can be easily found along the following lines. Let $\rho(a)=A\in \SU(2)$ and $\rho(b)=B\in \SU(2)$ such that $x=z$ and $y=zy-x$. Firstly, we can simultaneously conjugate $A$ and $B$ to write them as 
$$A = \left(\begin{array}{cc}r& 0\\ 0&\overline{r}\end{array}\right) \quad \textrm{ and } \quad B = \left(\begin{array}{cc}u& -\overline{v}\\ v&\overline{u}\end{array}\right)$$ 
where $|r|^2=1=|u|^2+|v|^2$. Since the trace of $A$ is $x=2 \textrm{ Re}(r):=2s$, the trace of $B$ is $y=2\textrm{ Re}(u)$ and the trace of 
$$AB = \left(\begin{array}{cc}ru& -r\overline{v}\\ \overline{r} v&\overline{r} \overline{u}\end{array}\right)$$ 
is $z= 2 \textrm{ Re}(ru) = 2 (\textrm{Re}(r) \textrm{Re}(u) - \textrm{Im}(r)\textrm{Im}(u))$, the fixed point equations $x=z$ and $y=zy-x$ become 
$$\textrm{ Re}(r) = \textrm{Re}(r) \textrm{Re}(u) - \textrm{Im}(r)\textrm{Im}(u)$$ 
and 
$$\textrm{ Re}(u) = 2\textrm{Re}(r)\textrm{ Re}(u) - \textrm{Re}(r).$$ 

Hence, if $s=\textrm{Re}(r) \neq 1/2$ and $\textrm{Im}(r)\neq 0$, then we can write 
$$\textrm{ Re}(u) = s/(2s-1)$$ 
and  
$$\textrm{Im}(u) = \frac{s}{\textrm{Im}(r)}\left(\frac{s}{2s-1}-1\right) = \frac{s}{\textrm{Im}(r)}\cdot\frac{1-s}{2s-1}.$$ 

Since the roles of $r$ and $\overline{r}$ can be exchanged and $|r|^2=1$, we can assume that $\textrm{Im}(r)=\sqrt{1-s^2}$. Also, it seems that $v$ can be taken arbitrarily within the circle $|v|^2 = 1-|u|^2$, so that we will choose $v=\sqrt{1-|u|^2}$. 

In this way, we get that 
$$A(s):=  \left(\begin{array}{cc}s+i\sqrt{1-s^2}& 0\\ 0&s-i\sqrt{1-s^2}\end{array}\right)$$ 
 and $B(s):=$

$$\left(\begin{array}{cc}\frac{s}{2s-1} + i \frac{s(1-s)}{(2s-1)\sqrt{1-s^2}}& -\sqrt{1-(\frac{s}{2s-1})^2-(\frac{s(1-s)}{(2s-1)\sqrt{1-s^2}})^2}\\ \sqrt{1-(\frac{s}{2s-1})^2-(\frac{s(1-s)}{(2s-1)\sqrt{1-s^2}})^2}&\frac{s}{2s-1} - i \frac{s(1-s)}{(2s-1)\sqrt{1-s^2}}\end{array}\right)$$
$$=\left(\begin{array}{cc}\frac{s}{2s-1} + i \frac{s(1-s)}{(2s-1)\sqrt{1-s^2}}& -\sqrt{1-\frac{2s^2}{4s^3-3s+1}}\\ \sqrt{1-\frac{2s^2}{4s^3-3s+1}}&\frac{s}{2s-1} - i \frac{s(1-s)}{(2s-1)\sqrt{1-s^2}}\end{array}\right),$$
so that 
$$\rho_s(a)=A(s), \quad \rho_s(b)=B(s)$$ 
describe a typical fixed point of the action of the cat map $M$ on the character variety $\X(F_2, \SU(2))$. 

\begin{remark} The trace of $[A(s), B(s)]$ is $\frac{2 (8 s^4 - 12 s^3 + 2 s^2 + 4 s -1)}{(1 - 2 s)^2}$. 
\end{remark}

Note that the fixed point at the origin $(0,0,0)$ in trace coordinates corresponds to the Pauli matrices 
$$A(0)=\left(\begin{array}{cc}i& 0\\ 0&-i\end{array}\right) \quad \textrm{and} \quad B(0)=\left(\begin{array}{cc}0& -1\\ 1&0\end{array}\right)$$ 
at $s=0$ and the parameter interval $s\in [-0.5405\dots, 0.2597\dots]$ around $s=0$ is interesting (as we will see). 

\subsection{Fixed points on $\X(F_2, \SU(3))$}

The second symmetric power 
$$\left(\begin{array}{cc}a& b\\ c&d\end{array}\right)^{\otimes 2} := \left(\begin{array}{ccc}a^2& ab & b^2\\ 2ac&ad+bc&2bd \\ c^2&cd&d^2\end{array}\right)$$ 
produces a family $\rho_s^{\otimes 2}$ of fixed points of the action of the cat map $M$ on $\X(F_2, \SU(3))$. 

The trace of $[A(s)^{\otimes 2}, B(s)^{\otimes 2}]$ is 
$$\frac{(256 s^8 - 768 s^7 + 704 s^6 + 64 s^5 - 448 s^4 + 192 s^3 + 24 s^2 - 24 s + 3)}{(1 - 2 s)^4}.$$

\begin{figure}[ht!]
\includegraphics[scale=0.4]{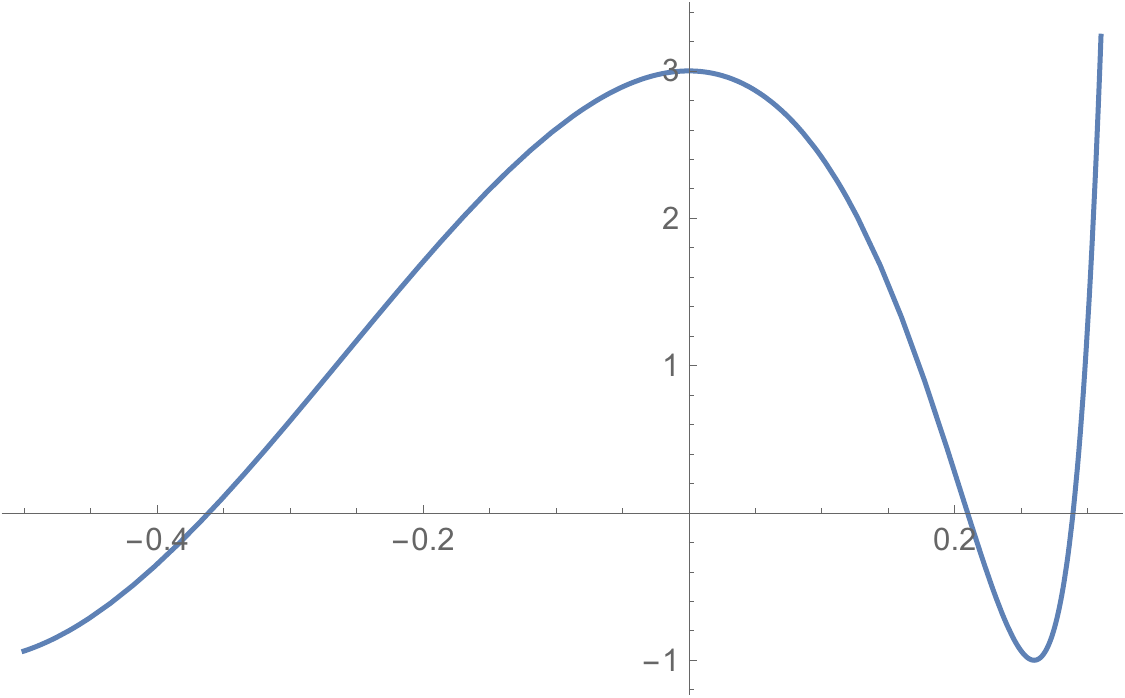} \caption{$\tr\left([A(s)^{\otimes 2}, B(s)^{\otimes 2}]\right)$}
\label{commgraph} 
\end{figure}

This real-valued function (Figure \ref{commgraph}) has a local extreme value $3$ at $s=0$. It increases monotonically on $[-0.5405\dots, 0]$ and decreases monotonically on $[0, 0.2597\dots]$ to attain the value $-1$ at the non-zero extremities of these intervals. In other words, the family $\rho_s^{\otimes 2}$ shows that the fixed points of the action of the cat map on $\X(F_2, \SU(3))$ cover the whole interval $\Delta\cap\R=[-1, 3].$

The matrices 
$$A(0)^{\otimes 2} = \left(\begin{array}{ccc}-1& 0 & 0\\ 0&1&0 \\ 0&0&-1\end{array}\right)$$ 
and 
$$B(0)^{\otimes 2} = \left(\begin{array}{ccc}0& 0 & 1\\ 0&-1&0 \\ 1&0&0\end{array}\right)$$ 
generate a \emph{reducible} representation $\rho_0^{\otimes 2}$ in the level set $$\mathfrak{L}_{(3,0)}\cong T^2/W\cong\C\mathrm{P}^2,$$ where $T\subset \SU(3)$ is a maximal torus and $W$ is the Weyl group.

\begin{remark}
In \cite{BL} it is shown that any hyperbolic element will act ergodically on character varieties of nilpotent groups.  In particular, $M$ acts ergodically on $\mathfrak{L}_{(3,0)}\cong\C\mathrm{P}^2.$
\end{remark}

\begin{remark}
Since the cat map $M$ acts via the quotient of the diagonal action of $M$ on $T^2$, we expect the spectrum of the action of $M$ restricted to the level set containing $\rho_s^{\otimes 2}$ to be close to $$\{(3+\sqrt{5})/2, (3-\sqrt{5})/2, (3+\sqrt{5})/2, (3-\sqrt{5})/2\}$$ for $s\approx 0$.  Numerical experiments have shown this in fact is the case (up to a phase shift).  It would be interesting to study the analogous spectrum for $s$ far from zero (i.e., for the trace of the commutator close to $-1$). 
\end{remark}

\begin{remark}
Numerical experiments suggest that the spectrum of $M$ varies (along $[-1,3]$) from having an interval where $M$ has an elliptic spectrum for $\ell\approx -1$ ($s\approx .25$) to having 4 hyperbolic and 2 elliptic eigenvalues for $\ell\approx 3$ ($s\approx 0$).
\end{remark}

\begin{remark}
We have checked (computationally) that there are no fixed points of the cat map $M$ along the arc of the boundary of $\Delta$ connecting $3$ to $3\omega=3e^{2\pi i/3}$ (likewise along the arc connecting $3$ to $3\overline{\omega}$).  This observation motivated our study of the spectrum of $M$ along $\Delta\cap\R=[-1,3]$.
\end{remark}

\begin{remark}
It is necessary that fixed points of $M$ on $\mathfrak{L}_{(\zeta, \eta)}$ have the form $(x,X,y,Y,x,X,y,-Y)$.  This follows by solving, with respect to the unitary coordinates discussed in Section \ref{sec-2}, the equation $$M\cdot (x,X,y,Y,z,Z,t,T)=(x,X,y,Y,z,Z,t,T)$$ where the left hand side simplifies to: $(z,Z,t - x y - X Y + y z - Y Z, T - X y + x Y + Y z + y Z, x + t z - y z - x y z - X Y z + y z^2 - T Z + X y Z - Y Z - x Y Z - 2 Y z Z - y Z^2, -X + T z - X y z - Y z + x Y z + Y z^2 + t Z + y Z - x y Z - X Y Z + 2 y z Z - Y Z^2, y, -Y).$
\end{remark}

\section{KAM theory on character varieties}\label{sec-6}

For this section denote $\X:=\X(F_2,\SU(3))$ and $\Y:=\X(F_2, \SU(2))$, and recall that $\pi_1(\Sigma_{1,1})\cong \langle a,b,c\ |\ c^{-1}=[a,b]\rangle\cong F_2$.

The natural actions of $\SL(2,\Z)$ on $\X$ and $\Y$ preserve the level sets of $\kappa_\X:\X\to \Delta$ and $\kappa_\Y:\Y\to [-2,2]$ given by $\kappa_\X([\rho])=\tr(\rho(c^{-1}))$ and $\kappa_\Y([\theta])=\tr(\theta(c^{-1}))$ respectively.  We note that in an earlier section the maps $\kappa_\X,\kappa_\Y$ were denoted $\mathfrak{b}_{1,1}$.

The level sets $\kappa_\X^{-1}(3e^{2\pi i/3})$ and $\kappa_\Y^{-1}(-2)$ are singleton sets so the action of any element in  $\SL(2,\Z)$ is trivially ergodic.  

The level sets $\kappa_\Y^{-1}(2)$ and $\kappa_\X^{-1}(3)$ are respectively $T_\Y^2/W_\Y$ and $T_\X^2/W_\X$, where $T_\Y\subset \SU(2)$ and $T_\X\subset \SU(3)$ are maximal tori with respective Weyl groups $W_\Y$ and $W_\X$. The action of $\SL(2,\Z)$ on $\kappa_\Y^{-1}(2)$ and $\kappa_\X^{-1}(3)$ factors through the diagonal action on $T^2_\Y$ and $T^2_\X$, and so the standard action on $\R^2/\Z^2$ and $\R^4/\Z^4$. It follows that any hyperbolic element of $\SL(2,\Z)$ acts ergodically on $\kappa_\Y^{-1}(2)$ and $\kappa_\X^{-1}(3)$.  See \cite{BL} for a more general discussion that includes these cases.

Our goal is to show that some hyperbolic elements of $\SL(2,\mathbb{Z})$, for example the cat map $\left(\begin{array}{cc} 2 & 1 \\ 1 & 1 \end{array}\right)$, act non-ergodically on $\kappa_\X^{-1}(\xi)$ for $\xi$ close to $-1$ despite the fact that $\SL(2,\Z)$ acts ergodically on $\kappa_\X^{-1}(\xi)$ for all boundary values $\xi$ by \cite{GLX}.

\subsection{KAM theory on a $\SU(2)$-character variety}\label{su2kam}

Let us briefly recall the strategy of Brown to derive the analogous statement in the $\SU(2)$ case which we discussed in Section \ref{sec-4}. We can write 
$$\Y=\{(x,y,z)\in[-2,2]^3 \ |\ -2\leq x^2+y^2+z^2-xyz-2\leq 2\}$$ 
in trace coordinates $(x,y,z)$, so that the elements of $\SL(2,\mathbb{Z})$ acts via polynomial automorphisms of $\mathbb{R}^3$.  For example, the cat map $\left(\begin{array}{cc} 2 & 1 \\ 1 & 1 \end{array}\right)$ acts via $(x,y,z)\mapsto (z,zy-x,z(zy-x)-y)$). Furthermore, we can ``blow up'' (or, more precisely, take the sphere of direction at) the point $\kappa^{-1}(-2)=\{(0,0,0)\}$ to get a $2$-sphere $\{(\dot{x},\dot{y},\dot{z})\ |\ \dot{x}^2+\dot{y}^2+\dot{z}^2=1\}$ where $\SL(2,\mathbb{Z})$ acts through a finite (octahedral) group (of order $24$).  For example, the cat map $\left(\begin{array}{cc} 2 & 1 \\ 1 & 1 \end{array}\right)$ acts via the element $(\dot{x},\dot{y},\dot{z})\mapsto (\dot{z},-\dot{x},-\dot{y})$ of order $3$. In particular, a hyperbolic element $\sigma$ of $\SL(2,\mathbb{Z})$ determines an one-dimensional semi-algebraic subset $$\textrm{Fix}(\sigma)=\{(x,y,z)\in \Y\ |\ \sigma(x,y,z)-(x,y,z)=0\}$$ of $\mathbb{R}^3$ consisting of its fixed points. The derivative of $\sigma|_{\kappa_\Y^{-1}(\xi)}$, at any point in $\textrm{Fix}(\sigma)\cap \kappa_\Y^{-1}(\xi)$, is the matrix $\sigma\in \SL(2,\mathbb{Z})$. The derivative of $\sigma|_{\kappa_\Y^{-1}(\xi)}$, at any point in $\textrm{Fix}(\sigma)\cap \kappa_\Y^{-1}(\xi)$ for $\xi$ close to $-2$, is a matrix close to an element of finite order in the sphere of directions.  For example, the subset of fixed points of the action of $\left(\begin{array}{cc} 2 & 1 \\ 1 & 1 \end{array}\right)$ on $\Y$ is the embedded line $$\{x=z, y=zy-x, z=z(zy-x)-y\}\cap \Y$$ and the derivatives of the actions of $\left(\begin{array}{cc} 2 & 1 \\ 1 & 1 \end{array}\right)$ on $\kappa_\Y^{-1}(\xi)$ at $\textrm{Fix}\cap \kappa_\Y^{-1}(\xi)$ vary from $\left(\begin{array}{cc} 2 & 1 \\ 1 & 1 \end{array}\right)$ at $\xi=2$ to a $2\times 2$ matrix of order $3$ at $\xi=-2$. Hence, a hyperbolic element $\sigma$ of $\SL(2,\mathbb{Z})$ ``usually'' produces a family of $2\times 2$ matrices in $\SL(2,\mathbb{R})$.  The family consists of derivatives of the action on $\kappa_\Y^{-1}(\xi)$ at $\textrm{Fix}\cap \kappa_\Y^{-1}(\xi)$ whose spectra vary from a pair of roots of unity (of orders $\leq 24$) at $\xi=-2$ to $\{\lambda,1/\lambda\}$ where $\lambda$ is a real number with $|\lambda|>1$.  Hence, we can apply KAM theory (Moser twist theorem) at those parameters $\xi$ where the spectrum is a pair of frequencies distinct from roots of unity of orders $\leq 6$ to conclude the non-ergodicity of the action of $\sigma$ on $\kappa_\Y^{-1}(\xi)$.  

For example, this strategy works for the cat map $\left(\begin{array}{cc} 2 & 1 \\ 1 & 1 \end{array}\right)$ as we can (and do!) compute the first Birkhoff invariant directly in {\it Mathematica} using the formulae at the end of Appendix \ref{appa}.  More precisely, we first compute the Birkhoff normal form of the cat map at fixed points near $\ell=-2$.  Using the normal form (its third order approximation is sufficient), we find the first Birkhoff invariant $\gamma_1\not=0$ (by showing $\alpha_2\not=0$, see Remark \ref{rema2}), which shows via KAM theory that the cat map is acting non-ergodically.  This computation is available on GitHub\footnote{\url{https://github.com/seanlawton/Non-ergodicity-on-SU-2-and-SU-3-character-varieties-of-the-once-punctured-torus}}.

\subsection{KAM theory on a $\SU(3)$-character variety}

In this section we outline the computational steps (implemented in {\it Mathematica} and available on GitHub) taken to prove that the cat map $M$ does not act ergodically on the relative character varieties $\mathfrak{L}_{(\zeta, \eta)}=\kappa_\X^{-1}(\zeta+i\eta)$ for $M$-fixed points $\zeta+i\eta$ near $-1$ (on the real line in the deltoid $\Delta$).

Here is an outline of the proof: (a) identify an elliptic fixed point in $\mathfrak{L}_{(\ell, 0)}$ for all $\ell$ in an interval $I$; (b) verify there exist elliptic fixed points that are irrational (in the terminology of \cite{EFK}) for some values of $\ell$ in $I$; (c) verify the non-planarity of the corresponding Birkhoff normal forms; and (d) apply the KAM theorem to get KAM tori.  The existence of such invariant tori prevents ergodicity.

The first step is to compute the cat map itself.  The map is polynomial in 9 variables on $\X(F_2,\SU(3))$.  Letting $x+iX=\tr_a$, $y+iY=\tr_b$, $z+iZ=\tr_{ab}$, $t+iT=\tr_{ab^{-1}}$, $u+iU=\tr_{aba^{-1}b^{-1}}$ be unitary coordinates on $\X(F_2,\SU(3))$ discussed in Section \ref{sec-2}, the cat map takes the form $(x,X,y,Y,z,Z,t,T,U)\mapsto(z, Z, t - x y - X Y + y z - Y Z, T - X y + x Y + Y z + y Z, x + t z - y z - x y z - X Y z + y z^2 - T Z + X y Z - Y Z - x Y Z - 2 Y z Z - y Z^2, -X + T z - X y z - Y z + x Y z + Y z^2 + t Z + y Z - x y Z - X Y Z + 2 y z Z - Y Z^2, y, -Y, U).$

Since we have already seen $(A(s), B(s))$ are fixed by the cat map for all $s\in[-1,3]$, we have a line of fixed points of the cat map given by $(x,X,y,Y,z,Z,t,T,U)=$ $$\left(-1 + 4 s^2, 0, -1 + \frac{4 s^2}{(1 - 2 s)^2}, 0, -1 + 4 s^2, 0, -1 + \frac{4 s^2}{(1 - 2 s)^2}, 0, 0\right).$$

For each $s$, the fixed point lies in $\mathfrak{L}_{(\ell,0)}=\kappa^{-1}_\X(\ell)$ for $\ell=$ $$\frac{((-3 + 4 s (3 - 6 s^2 + 4 s^3)) (-1 + 4 s (1 + 2 (-1 + s) s (-1 + 2 s))))}{(1 - 2 s)^4}.$$

The above formula comes from $u=\ell=P/2$.  Note also that $Q=\ell^2$ and consequently $P^2-4Q=0$.  We will use this implicit equation below to determine the third jet of the cat map in a chart.  For $s\in [-1,3)$, these fixed points come from irreducible representations in $\SU(2)$ composed with the second symmetric tensor (also an irreducible representation).  So they are irreducible representations and hence are smooth points in $\mathfrak{L}_{(\ell,0)}$ for $s\in [-1,3)$.  The moduli spaces $\mathfrak{L}_{(\ell,0)}$ are 6 dimensional singular varieties, but their smooth loci are full measure on which $\SL(2,\Z)$ acts smoothly.

We need to find a smooth chart of our elliptic fixed point and write the Taylor series of the cat map (up to third order).  Given the dimension of $\mathfrak{L}_{(\ell,0)}$ is 6, but naturally described in terms of 8 variables, we need to eliminate (locally) two variables.  We do this is two steps.  The first will be to use $P/2=\ell$ to explicitly solve for one variable and then we will use $P^2-4Q=0$ to implicitly solve for another.  There are many choices involved, and some but not all, lead to singular charts at our fixed points.  So we take care to avoid such choices.

Indeed, solving for $t$ in $P/2=\ell$, we find: \begin{align*}t&=xy+XY\\&-\bigg(\frac{16 (s-1) s \left(8 (s-1) (s (4 (s-1) s-1)+3)
   s^2+3\right)}{(1-2 s)^4}+\frac{6}{(1-2 s)^4}\\&-T^2+2 T (Xy-xY)-x^2 \left(Y^2+1\right)+2 x (X y Y+y z+Y Z)\\&-X^2
   \left(y^2+1\right)+2 X y Z-2 X Y z-y^2-Y^2-z^2-Z^2+3\bigg)^{1/2}.
\end{align*}

\begin{remark}
From the equation $P/2=\ell$, there are exactly two such solutions for any of the 8 variables we choose.  The solution substitutions are smooth if and only if the polynomial under the radical is non-zero.  Along the line of fixed points we are considering, the substitution for $t$ above (and also $z$) is in fact smooth (unlike some others like for $T$ or $Z$).
\end{remark}

Using this expression for $t$ we can explicitly write the Taylor expansion (centered at the fixed points) up to third order.  For example, at $s=.249$ ($\ell\approx -1$), the third order jet of $t$ centered at the fixed point (determined again by $s=.249$) is:\newline

\begin{minipage}{0.9\textwidth}$
-0.0158728 + 35.9596 T^2 - 27.4865 (0.751996 + x) - 
 106566. T^2 (0.751996 + x) + 27172.4 (0.751996 + x)^2 - 
 8.05253\times 10^7 (0.751996 + x)^3 + 1.14156 T X - 
 3383. T (0.751996 + x) X + 35.9686 X^2 - 
 106593. (0.751996 + x) X^2 - 21.6578 (0.0158728 + y) - 
 81099.4 T^2 (0.0158728 + y) + 
 41359. (0.751996 + x) (0.0158728 + y) - 
 1.83843\times 10^8 (0.751996 + x)^2 (0.0158728 + y) - 
 2790.3 T X (0.0158728 + y) - 81123.2 X^2 (0.0158728 + y) + 
 15752.2 (0.0158728 + y)^2 - 
 1.39954\times 10^8 (0.751996 + x) (0.0158728 + y)^2 - 
 3.55259\times 10^7 (0.0158728 + y)^3 - 54.0829 T Y + 
 160490. T (0.751996 + x) Y - 52.9413 X Y + 
 162822. (0.751996 + x) X Y + 121973. T (0.0158728 + y) Y + 
 124071. X (0.0158728 + y) Y + 56.2946 Y^2 - 
 166991. (0.751996 + x) Y^2 - 126961. (0.0158728 + y) Y^2 - 
 27.4707 (0.751996 + z) - 106566. T^2 (0.751996 + z) + 
 54274. (0.751996 + x) (0.751996 + z) - 
 2.41366\times 10^8 (0.751996 + x)^2 (0.751996 + z) - 
 3383. T X (0.751996 + z) - 106593. X^2 (0.751996 + z) + 
 41357. (0.0158728 + y) (0.751996 + z) - 
 3.67527\times 10^8 (0.751996 + x) (0.0158728 + y) (0.751996 + z) -  1.39954\times 10^8 (0.0158728 + y)^2 (0.751996 + z) + 
 160275. T Y (0.751996 + z) + 163034. X Y (0.751996 + z) - 
 166829. Y^2 (0.751996 + z) + 27172.4 (0.751996 + z)^2 - 
 2.41366\times 10^8 (0.751996 + x) (0.751996 + z)^2 - 
 1.83843\times 10^8 (0.0158728 + y) (0.751996 + z)^2 - 
 8.05253\times 10^7 (0.751996 + z)^3 + 1.14156 X Z - 
 3383. (0.751996 + x) X Z - 2790.3 X (0.0158728 + y) Z + 
 54.0829 Y Z - 160490. (0.751996 + x) Y Z - 
 121973. (0.0158728 + y) Y Z - 3383. X (0.751996 + z) Z - 
 160275. Y (0.751996 + z) Z + 35.9596 Z^2 - 
 106566. (0.751996 + x) Z^2 - 81099.4 (0.0158728 + y) Z^2 - 
 106566. (0.751996 + z) Z^2.$
 \\
 \end{minipage}
 
To get the third jet for $z$ we use $H:=(P/2)^2-Q=0$ implicitly.  First we substitute the third jet of $t$ into $H$ making $H$ a third order approximation in the 7 variables $x,X,y,Y,z,Z,T$.  Then we calculate all partial derivatives of $z$ implicitly at fixed points in terms of $s$ using $H$.  We continue this until we have the third jet of $z$ in terms of only the 6 variables $x,X,y,Y,Z,T$ up to third order centered at fixed points.  For example, at $s=.249$ the third order jet of $z$ is:\newline

\begin{minipage}{0.9\textwidth}$
-1.50399 + 1.28663 T^2 - x - 5.16029 T^2 (0.751996 + x) + 
 2.67381 (0.751996 + x)^2 - 10.7239 (0.751996 + x)^3 + 
 0.0844526 T X + 0.0167138 T (0.751996 + x) X + 1.22118 X^2 - 
 4.71443 (0.751996 + x) X^2 - 0.751996 (0.0158728 + y) - 
 0.799416 T^2 (0.0158728 + y) + 
 2.0107 (0.751996 + x) (0.0158728 + y) + 
 2.78895 (0.751996 + x)^2 (0.0158728 + y) - 
 6.31043 T X (0.0158728 + y) - 2.56426 X^2 (0.0158728 + y) + 
 0.673811 (0.0158728 + y)^2 + 
 5.45915 (0.751996 + x) (0.0158728 + y)^2 + 
 0.681032 (0.0158728 + y)^3 - 1.93507 T Y + 
 14.6832 T (0.751996 + x) Y - 1.8878 X Y - 
 0.0867515 (0.751996 + x) X Y + 0.902415 T (0.0158728 + y) Y + 
 3.0885 X (0.0158728 + y) Y + 1.94383 Y^2 - 
 13.0016 (0.751996 + x) Y^2 - 2.78224 (0.0158728 + y) Y^2 + 
 0.043608 T Z - 0.530329 T (0.751996 + x) Z + 0.0387346 X Z - 
 0.160995 (0.751996 + x) X Z + 1.16741 T (0.0158728 + y) Z - 
 7.39926 X (0.0158728 + y) Z + 1.7915 Y Z - 
 14.4187 (0.751996 + x) Y Z - 4.96467 (0.0158728 + y) Y Z + 
 1.22016 Z^2 - 5.07144 (0.751996 + x) Z^2 - 
 2.46712 (0.0158728 + y) Z^2.$
 \\
 \end{minipage}

To get a local representation of the cat map $M$ at a fixed point (in terms of $s$) we then substitute the third jet of $t$ to reduce to 7 variables, and then we substitute the third jet of $z$ to reduce to 6 variables. The resulting expression maps a smooth open neighborhood of a fixed point (diffeomorphic to $\R^6$) into $\R^8$.  We then project onto $\R^6$ by forgetting the variables we have eliminated ($t,z$).  This projection is smooth (whenever the jets are defined) since the eliminated variables are locally a graph.  

We experimentally find values of $s$ between .239 and .249 (that is, $\ell$ near $-1$) where the spectrum is elliptic and the chart is smooth.  This implies there exists an interval between .239 and .249  with the same properties by continuity \cite{HaMa}.

Now to apply KAM theory as described in Appendix \ref{appb}, we need to put the third jet of the cat map in its Birkhoff normal form.  So we first translate the codomain and domain of our chart so the fixed points are at the origin (in both), and then diagonalize the jets (so the Jacobian is diagonal).  This amounts to finding the matrix $C_0$ in Appendix \ref{appb} after doing a change of variables making the fixed points the origin.  

Following the recipe in Appendix \ref{appb} we then make the formal change of coordinates tangent to the identity.  This amounts to making substitutions of the form $x-x_0\mapsto \frac{\xi_1+\eta_1}{2}$ and $X-X_0\mapsto\frac{\xi_1-\eta_1}{2i}$ etc.  Once these substitutions have been made for all  6 variables in terms of  $\xi_1,\eta_1,...,\xi_3,\eta_3$ we obtain the Birkhoff normal form of the cat map (up to third order).

Once we have the Birkhoff normal form, we pick out the appropriate coefficients from the jets to calculate coefficients of the Birkhoff normal form (Equation \eqref{eq:higheralpha}) whose determinant detects ``torsion''.  We calculate this at fixed points finding it non-zero and hence showing the twist condition holds.

For example, at $s=.249$, the relevant coefficients of the Birkhoff normal form are: $(\alpha_{jk})_{1\leq j,k\leq 3}=$
{\scriptsize $$
\left(
\begin{array}{ccc}
 0.00552244\, -0.0340402 i & 0.0107941\, -0.000895037 i & 1.27133\,
   +2.0689 i \\
 -0.200044-0.525768 i & -0.327311-0.329913 i & -0.800469-0.841658 i
   \\
 -4.01094-2.67688 i & -8.79221-8.77867 i & 250.545\, +281.496 i \\
\end{array}
\right),$$} whose determinant is $-20.077 - 0.73655 i$.

To complete the verification of the KAM criterion as per \cite[Theorem 1.4]{EFK}, we check non-planarity (see the final subsection in Appendix \ref{appb}).  

Although it suffices to verify an infinitesimal non-planarity condition at a single point, to reduce the number of computations, we instead verify {\it generic} non-planarity by checking that a certain determinant on four points of the curve is non-zero.  This ensures that the curve has four points not lying on a plane.  Since the curve is analytic we conclude that the zero set of the determinant of derivatives is discrete (hence has no accumulation points).  In short, since the curve is analytic and has four points not contained in a plane, the point-wise infinitesimal non-planarity condition holds at all but possibly finitely many points. 

This is a numeric criterion which we verify directly in {\it Mathematica} using $s$ values of $.239, .24, .241, .242.$  With this complete, we thus know that there is an invariant torus and hence non-ergodicity.

The computations we have described above prove:

\begin{theorem}\label{main} Let $M\in\SL(2,\Z)$ be the cat map. Then, there exists $\ell\in \Delta$ such that $M$ does not act ergodically on $\kappa_\X^{-1}(\ell)\subset \X(F_2,\SU(3))$. 
\end{theorem}

\appendix

\section{Elliptic fixed points of surface diffeomorphisms}\label{appa} 

We follow the discussion of Birkhoff normal forms in \cite[\S 23]{SiegelMoser}. 

Let $x_1=f(x,y)$, $y_1=g(x,y)$ be the coordinate functions of a real-analytic surface diffeomorphism $S$ with an elliptic fixed point at the origin $(0,0)$.  Let
$$f(x,y) = a x + by + \cdots, \text{and } g(x,y) = cx+dy+\cdots$$ 
where $\left(\begin{array}{cc} a & b \\ c& d\end{array}\right)\in \SL(2,\mathbb{R})$ is a matrix with elliptic eigenvalues. Note: the ``$\cdots$'' in the expression for $f,g$ stands for higher-order terms. 

Let us write $z=x+iy$ and $z_1=x_1+iy_1$, so that $S$ becomes $S(z)=z_1$. We want to simplify the Taylor expansion of $S$ via an appropriate change of variables 
$$x=\phi(\xi,\eta) = \alpha \xi+\beta\eta+\cdots, \quad y=\psi(\xi,\eta) = \gamma\xi+\delta\eta+\cdots,$$ 
$x_1=\phi(\xi_1,\eta_1)$, and  $y_1=\psi(\xi_1,\eta_1)$. In other words, $z=C\zeta$, $z_1=C\zeta_1$, $\zeta_1=\xi_1+i\eta_1$, $\zeta=\xi+i\eta$, and $\alpha\delta-\beta\gamma\neq 0$.

\subsection{Preliminary reduction} Take $C_0\in \SL(2,\mathbb{C})$ such that 
$$C_0^{-1} \left(\begin{array}{cc} a & b \\ c& d\end{array}\right) C_0 = \left(\begin{array}{cc} \lambda & 0 \\ 0& \mu\end{array}\right)$$ 
where $\mu=\overline{\lambda}$ and $|\lambda|=ad-bc=1$. Note that the change of variables by $C_0$ converts $z_1=Sz$ into 
$$\zeta_1=C_0^{-1}z_1 = C_0^{-1}SC_0\zeta = T\zeta$$ 
where $\xi_1 = p(\xi,\eta) = \lambda\xi+\cdots$, $\eta_1 = q(\xi,\eta)=\mu\eta+\cdots$. Since $S$ is real, we also have that $\overline{p}(\xi,\eta) = q(\eta,\xi)$. 

\subsection{Birkhoff normal forms} We want to build a (formal) change of coordinates $C$ tangent to the identity, say 
$$x=\phi(\xi,\eta) = \xi+\sum\limits_{k=2}^{\infty}\phi_k, \quad y=\psi(\xi,\eta) = \eta+\sum\limits_{k=2}^{\infty}\psi_k,$$ 
where $\phi_k$ and $\psi_k$ are homogenous of degree $k$ in $\xi$ and $\eta$, such that $U=C^{-1}TC$ has the \emph{Birkhoff normal form} 
$$\xi_1 = u\xi, \quad \eta_1=v\eta, \quad u=\sum_{k=0}^{\infty}\alpha_{2k} (\xi\eta)^k, \quad v=\sum\limits_{k=0}^{\infty}\beta_{2k}(\xi\eta)^k.$$ 

\begin{remark} As it is shown in \S23 of Siegel--Moser's book \cite{SiegelMoser}, if $S$ preserves area in the sense that $f_xg_y-f_yg_x=1$, then the formal series defining $C$ as above is unique (and, moreover, $C$ preserves area) whenever $\lambda$ is not a root of unity. In the sequel, we will review the construction of $C$ in order to find the first few terms of its Taylor expansion. 
\end{remark}

\begin{remark}\label{rema2} If we write $u=e^{iw}$, $v=e^{-iw}$, $w(\omega) = \sum\limits_{k=0}^{\infty}\gamma_k\omega^k$, where $\omega=\xi\eta$, the \emph{Birkhoff invariants} $\gamma_k$ are real, and $e^{i\gamma_0}=\lambda$, $-\pi<\gamma_0<\pi$, then $$u=\lambda e^{i\sum\limits_{k=1}^{\infty}\gamma_k\omega^k} = \lambda(1+i\gamma_1\omega+O(\omega^2)),$$ so that the first Birkhoff invariant $\gamma_1$ does not vanish if and only if $\alpha_2\neq0$. 
\end{remark} 

By definition, we want to solve the functional equations 
$$\phi(u\xi,v\eta) = p(\phi(\xi,\eta),\psi(\xi,\eta)) \quad \textrm{ and } \quad \psi(u\xi,v\eta)=q(\phi(\xi,\eta),\psi(\xi,\eta)).$$ 

By comparing the linear terms, we impose $\alpha_0=\lambda$, $\beta_0=\mu$. Also, we set $\alpha_{\ell}=\beta_{\ell}=0$ for each $\ell\geq1$ that is odd. 

Assuming that $\phi_{\ell}$, $\psi_{\ell}$, $\alpha_{\ell-1}$, $\beta_{\ell-1}$, $\ell<k$ were determined for a certain $k\geq 2$ (for example $k=2$), let us determine $\phi_{k}$, $\psi_{k}$, $\alpha_{k-1}$, $\beta_{k-1}$.

By comparing the terms of degree $k$ in our functional equations, we get that 
$$\phi_k(\lambda\xi,\mu\eta)+\alpha_{k-1}(\xi\eta)^{(k-1)/2}\xi = \lambda\phi_k(\xi,\eta)+\cdots$$ 
and 
$$\psi_k(\lambda\xi,\mu\eta)+\beta_{k-1}(\xi\eta)^{(k-1)/2}\eta = \mu\psi_k(\xi,\eta)+\cdots.$$ 
Note: in the previous expressions ``$\cdots$'' stands for homogeneous terms of degree $k$ already known. 

Next, let us write $\phi_k(\xi,\eta)=\sum\limits_{\ell=0}^k a_{k;\ell} \xi^{k-\ell}\eta^{\ell}$ and $\psi_k(\xi,\eta)=\sum\limits_{\ell=0}^k b_{k;\ell} \xi^{k-\ell}\eta^{\ell}$, so that 
$$\phi_k(\lambda\xi,\mu\eta) - \lambda \phi_k(\xi,\eta) = \sum\limits_{\ell=0}^k a_{k;\ell}(\lambda^{k-\ell}\mu^{\ell}-\lambda) \xi^{k-\ell}\eta^{\ell}$$
and 
$$\psi_k(\lambda\xi,\mu\eta) - \mu \psi_k(\xi,\eta) = \sum\limits_{\ell=0}^k b_{k;\ell}(\lambda^{k-\ell}\mu^{\ell}-\mu) \xi^{k-\ell}\eta^{\ell}.$$ 
Since $\lambda\mu=1$, one has $\lambda^{k-\ell}\mu^{\ell}-\lambda = \lambda(\lambda^{k-2\ell-1}-1)$ and $\lambda^{k-\ell}\mu^{\ell}-\mu = \lambda^{-1}(\lambda^{k-2\ell+1}-1)$. Hence, if $\lambda$ is not a ``small'' root of unity (so that $\lambda^{k-2\ell\pm1}=1$ if and only if $k=2\ell\mp1$), then we can determine uniquely $\alpha_{k-1}$, $\beta_{k-1}$, $a_{k;\ell}$ for $\ell\neq \frac{k-1}{2}$, and $b_{k;\ell}$ for $\ell\neq\frac{k+1}{2}$ from the previous two systems of equations. 

Therefore, it remains only to determine $a_{k;h}$ and $b_{k;h+1}$ when $k=2h+1$ is odd. Indeed, observe that the terms of degree $<n-1$ of 
$$\phi_{\xi}-\psi_{\eta} = \sigma(\xi,\eta) \quad \textrm{and} \quad \phi_{\xi}\psi_{\eta}-\phi_{\eta}\psi_{\xi}-1=\tau(\xi,\eta)-1$$ 
do not contain powers of $\omega=\xi\eta$ for $n=2$ and, in general, this is also trivially true for $n+1$ when it is true for $n$ even. Moreover, for $k=2h+1$, the coefficient of $\omega^h=(\xi\eta)^h$ in $\sigma(\xi,\eta)$ is $(h+1)(a_{k;h}-b_{k;h+1})$, so that we want 
$$a_{k;h}=b_{k;h+1}.$$
Furthermore, the terms of degree $k-1$ in $\tau(\xi,\eta)$ are $(\phi_k)_{\xi}+(\psi_k)_{\eta}$ and a polynomial whose coefficients are already known. In particular, the vanishing of the coefficient of $\omega^h$ in $\tau$ determines $(h+1)(a_{k;h}+b_{k;h+1})$, and we are done. 

\begin{remark} When $S$ preserves area, the fact that we avoid powers of $\omega=\xi\eta$ in $\sigma(\xi,\eta)$ and $\tau(\xi,\eta)$ actually shows that the formal change of variables $C$ preserves area (i.e., $\tau(\xi,\eta)\equiv 1$).  
\end{remark}

\subsection{Computation of the first Birkhoff invariant} Let us now determine $\alpha_2$ by explicitly working out the argument above.

We take $k=2$, so that $\alpha_0=\lambda$, $\beta_0=\mu$, $\alpha_1=0=\beta_1$ and $\phi_1(\xi,\eta)=\xi$, $\psi_1(\xi,\eta)=\eta$. We write the coordinates of $T$ as 
$$p(\xi,\eta) = \lambda\xi+p_{2;0}\xi^2+p_{2;1}\xi\eta+p_{2;2}\eta^2+\cdots$$
and 
$$q(\xi,\eta) = \mu\eta+q_{2;0}\xi^2+q_{2;1}\xi\eta+q_{2;2}\eta^2+\cdots$$ 

The relevant functional equations become 
$$\phi_2(\lambda\xi,\mu\eta) = \lambda\phi_2(\xi,\eta)+p_{2;0}\xi^2+p_{2;1}\xi\eta+p_{2;2}\eta^2$$ 
and 
$$\psi_2(\lambda\xi,\mu\eta) = \mu\psi_2(\xi,\eta)+q_{2;0}\xi^2+q_{2;1}\xi\eta+q_{2;2}\eta^2.$$ 
Thus, the resulting equations 
$$a_{2;0}(\lambda^2-\lambda)\xi^2+a_{2;1}(\lambda\mu-\lambda)\xi\eta+a_{2;2}(\mu^2-\lambda)\eta^2 = p_{2;0}\xi^2+p_{2;1}\xi\eta+p_{2;2}\eta^2$$ 
and 
$$b_{2;0}(\lambda^2-\lambda)\xi^2+b_{2;1}(\lambda\mu-\mu)\xi\eta+b_{2;2}(\mu^2-\mu)\eta^2 = q_{2;0}\xi^2+q_{2;1}\xi\eta+q_{2;2}\eta^2$$ 
determine $\phi_2$ and $\psi_2$. 

Next, given that $\alpha_0=\lambda$, $\beta_0=\mu$, $\alpha_1=0=\beta_1$, $\phi_1(\xi,\eta)=\xi$, $\psi_1(\xi,\eta)=\eta$, and $\phi_2$, $\psi_2$ were determined, we take $k=3$ to get the functional equations 
\begin{eqnarray*}
\phi_3(\lambda\xi,\mu\eta)+\alpha_{2}(\xi\eta)\xi &=& \lambda\phi_3(\xi,\eta) +\big[p_{2;0}(\xi+\phi_2(\xi,\eta)+\cdots )^2 \\&+&\ p_{2;1}(\xi+\phi_2(\xi,\eta)+\cdots )(\eta+\psi_2(\xi,\eta)+\cdots )\\&+& p_{2;2}(\eta+\psi_2(\xi,\eta)+\ \cdots )^2\big]_{\textrm{deg }3}
\\&+& p_{3;0}\xi^3+p_{3;1}\xi^2\eta+p_{3;2}\xi\eta^2+p_{3;3}\eta^3
\end{eqnarray*} 
and 
\begin{eqnarray*}
\psi_3(\lambda\xi,\mu\eta)+\beta_{2}(\xi\eta)\eta &=& \mu\psi_3(\xi,\eta)
+\big[q_{2;0}(\xi+\phi_2(\xi,\eta)+\cdots)^2 \\&+& q_{2;1}(\xi+\phi_2(\xi,\eta)+\cdots)(\eta+\psi_2(\xi,\eta)+\cdots)\\&+& q_{2;2}(\eta+\psi_2(\xi,\eta)+\cdots)^2\big]_{\textrm{deg }3} \\ 
&+& q_{3;0}\xi^3+q_{3;1}\xi^2\eta+q_{3;2}\xi\eta^2+q_{3;3}\eta^3
\end{eqnarray*} 

By writing the identities for the coefficient of the term $\xi^2 \eta$ in the first formula, we have
$$
\lambda (\lambda \mu -1) a_{3,1} + \alpha_2 = 2 p_{2,0} a_{2,1} + p_{2,1} (b_{2,1} + a_{2,0}) + 2 p_{2,2} b_{2,0} + p_{3,1}
$$
where $a_{2,0}$, $a_{2,1}$ is given in terms of $p_{2,0}$, $p_{2,1}$ and $b_{2,0}$, $b_{2,1}$ are given in terms of $q_{2,0}$, $q_{2,1}$
respectively. In fact,
$$
a_{2,0} =  \frac{p_{2,0}}{ \lambda (\lambda-1)}\,, \quad a_{2,1} =  \frac{p_{2,1}}{ \lambda (\mu-1)}$$ and
$$b_{2,0} =  \frac{q_{2,0}}{ \lambda (\lambda-1)}\,, \quad b_{2,1} =  \frac{q_{2,1}}{ \mu (\lambda-1)} .
$$
In conclusion, since $\lambda \mu=1$, the above equations, while they are not sufficient to determine $a_{3,1}$, they do determine the coefficient
$\alpha_2$ of the Birkhoff normal form:
$$
\alpha_2 =   \frac{ 2 p_{2,0}  p_{2,1}}{ \lambda (\mu-1)}  + \frac{p_{2,1}}{ \lambda \mu (\lambda-1)} ( \lambda q_{2,1}   +  \mu p_{2,0}) +  \frac{ 2 p_{2,2}q_{2,0}}{ \lambda (\lambda-1)} + p_{3,1}.
$$ 

\section{Elliptic fixed points of symplectomorphisms}\label{appb}

The subsequent discussion is based on the descriptions of Birkhoff normal forms in \cite{EFK}, \cite{K}, \cite{M} and \cite{SiegelMoser}.

Let $f$ be a real-analytic symplectomorphism with an elliptic fixed point at the origin $0\in\mathbb{R}^{2d}$, that is, the spectrum of $Df(0)$ has the form  
$$\{e^{\pm2\pi i \omega_j}:1\leq j\leq d\}$$ 
for some frequency vector $\omega=(\omega_1,\dots,\omega_d)\in [0,1/2]^d$. 

We take $C_0\in \SL(2d,\mathbb{C})$ such that $$C_0^{-1}Df(0)C_0 = \textrm{diag}(\dots, e^{2\pi i\omega_j}, e^{-2\pi i\omega_j},\dots),$$ and $T = C_0^{-1} f C_0$ has the form: 
$$\xi_j\mapsto p_j(\xi ,\eta) = \lambda_j\xi_j +O_2(\xi, \eta)  \quad \textrm{ and } \quad \eta_j\mapsto q_j(\xi,\eta) = \mu_j\eta_j+ O_2(\xi, \eta) $$ 
where $\lambda_j=e^{2\pi i\omega_j}$, $\mu_j=\overline{\lambda_j}$, for all $j=1, \dots, d$  and $O_2(\xi, \eta)$ stand for higher-order terms in $
\xi:= (\xi_1, \dots, \xi_d)$ and $\eta:= (\eta_1, \dots, \eta_d)$.

\subsection{Birkhoff normal form} We want to build a (formal) change of coordinates $C$ tangent to the identity: 
$$x_j =\phi_j(\xi , \eta) = \xi_j + \sum\limits_{k=2}^{\infty}\phi_{j,k}(\xi,\eta)\ \text{ and } \ y_j =\psi_j(\xi, \eta) = \eta_j + \sum\limits_{k=2}^{\infty}\psi_{j,k}(\xi,\eta),$$ 
where $\phi_{j,k}$ and $\psi_{j,k}$ are homogeneous of degree $k$ in $(\xi, \eta)$, such that $C^{-1}TC$ has the \emph{Birkhoff normal form}:
$$\xi_j\mapsto u_j\xi_j \ \textrm{ and } \ \eta_j\mapsto v_j\eta_j,$$ 
where $u_j=e^{i\partial_{r_j}B}$, $v_j = e^{-i\partial_{r_j}B}$, and $$B(r) = 2\pi\langle\omega, r\rangle + \sum \frac{1}{2}b_{mn} r_m r_n +O_3(r)$$ is a (formal) power series on $r=(r_1,\dots, r_d)$, with $r_j:=\xi_j\eta_j$ with the convention that $b_{mn}=b_{nm}$. 

\begin{remark}
For later use, note that $u_j = \lambda_j(1 + i\sum b_{jk} r_k + O_2(r)) = \lambda_j+\sum \alpha_{jk} r_k+O_2(r)$, where $\alpha_{jk}=i\lambda_j b_{jk}$.  In order to compute the Birkhoff invariants $(b_{jk})$ and establish the desired twist condition (see Subsection  \ref{subsectwist}) it is therefore enough to compute the coefficients $(\alpha_{jk})$ of the Birkhoff normal form (see Formula \eqref{eq:higheralpha} below).
\end{remark}

\begin{remark} Given a real-analytic function $F$ of $(\xi, \eta)$, we denote by $F_{k}^{a_1,\dots,a_n|b_1,\dots,b_m}$, $k=n+m$, the normalized\footnote{In view of the symmetries given by permutations of the variables each coefficient of an unordered monomial $\xi_{a_1}\cdots\xi_{a_n}\eta_{b_1}\cdots\eta_{b_m}$ should be normalized dividing it by the factor
$n! m! /\mu_{a,1}! \cdots \mu_{a, k_a} !  \mu_{b,1} ! \cdots \mu_{b,k_b}!$  with $\mu_{a, 1}, \dots \mu_{a, k_a}$ and $\mu_{b, 1} \cdots \mu_{b, k_b}$  the list of multiplicities of $a_1, \dots a_n$ and $b_1, \dots, b_m$, respectively.}   coefficient of the ordered monomial $\xi_{a_1}\cdots\xi_{a_n}\eta_{b_1}\cdots\eta_{b_m}$ in the Taylor expansion of $F$, with the convention that $\xi_0=\eta_0=1$. 
\end{remark} 

\subsection{Computation of the first Birkhoff invariants} Let $$ u\xi  = (u_1\xi_1, \dots u_d \xi_d)$$ and similarly $v\eta=  (v_1\eta_1, \dots v_d \eta_d)$.  By definition, we want to solve the functional equations:
$$\phi_j(u \xi ,v\eta) = p_j(\phi(\xi,\eta),\psi(\xi, \eta)) \ \textrm{ and } \ \psi_j(u\xi,v \eta) = q_j(\phi(\xi,\eta),\psi(\xi, \eta)).$$

Let us first look at the terms of degree $\leq 2$. Let the symbol $O_3(\xi, \eta)$ stand for terms of degree at least $3$ in the vector $(\xi, \eta)$. Note that 
\begin{eqnarray*} 
\phi_j(u \xi,v\eta) &=& u_j\xi_j + \sum_{m,n=1}^d \phi_{j,2}^{m,n|0} u_m u_n \xi_m \xi_n + \sum_{m,n=1}^d \phi_{j,2}^{m|n} u_m v_n \xi_m \eta_n  \\ &+&\sum_{m,n=1}^d \phi_{2,j} ^{0|m,n} v_m v_n \eta_m\eta_n + O_3(\xi,\eta) \\ 
&=& \lambda_j\xi_j + \sum_{m,n=1}^d \phi_{j,2}^{m,n|0} \lambda_m\lambda_n \xi_m\xi_n + \sum_{m,n=1}^d \phi_{j,2}^{m|n} \lambda_m \mu_n \xi_m \eta_n  \\ &+&\sum_{m,n=1}^d \phi_{j,2}^{0|m,n} \mu_m\mu_n \eta_m\eta_n + O_3(\xi,\eta)
\end{eqnarray*} 
and 
\begin{eqnarray*}
p_j(\phi(\xi,\eta),\psi(\xi, \eta)) &=& \lambda_j \phi_j(\xi,\eta) + \sum_{m,n=1}^d  p_{j, 2}^{ m,n|0} \phi_m(\xi,\eta) \phi_n(\xi,\eta)  \\ &+& \sum_{m,n=1}^d p_{j,2}^{m|n} \phi_m(\xi,\eta) \psi_n(\xi,\eta) 
\\&+& \sum_{m,n=1}^d p_{j,2}^{0|m,n} \psi_m(\xi,\eta)\psi_n(\xi,\eta) + O_3(\xi,\eta) \\ 
&=& \lambda_j\xi_j \\&+&\!\!\!\lambda_j\! \left(\sum_{m,n=1}^d \phi_{j,2}^{m,n|0} \xi_m\xi_n + \phi_{j,2}^{m|n} \xi_m \eta_n+\phi_{j,2}^{0|m,n} \eta_m\eta_n \right) 
\\ &+& \sum_{m,n=1}^d p_{j,2}^{m,n|0} \xi_m\xi_n + \sum_{m,n=1}^d p_{j,2}^{m|n} \xi_m \eta_n \\&+& \sum_{m,n=1}^d p_{j,2}^{0|m,n} \eta_m\eta_n + O_3(\xi,\eta).
\end{eqnarray*} 

Hence, we can determine $\phi_{j,2}$ (resp., $\psi_{j,2}$) in terms of $ p_{j,2}$ (resp., $q_{j,2}$) \emph{provided} that the non-resonance conditions 
$$\lambda_j\neq \lambda_m\lambda_n, \lambda_m\mu_n, \mu_m\mu_n$$ 
are satisfied for all $1\leq j, m, n\leq d$.   In fact we have the formulae:
\begin{equation}
\label{eq:phi_{j,2}}
\phi_{j,2}^{m,n|0} =  \frac{p_{j,2}^{m,n|0}}{ \lambda_m\lambda_n -\lambda_j }\,,  \  \phi_{j,2}^{m|n} =  \frac{p_{j,2}^{m|n}}{ \lambda_m\mu_n -\lambda_j }
\ \text{ and } \ \phi_{j,2}^{0|m,n} =  \frac{p_{j,2}^{0|m,n}}{ \mu_m\mu_n -\lambda_j } \,.
\end{equation}
and similarly (or by complex conjugation):
\begin{equation}
\label{eq:psi_{j,2}}
\psi_{j,2}^{m,n|0} =  \frac{q_{j,2}^{m,n|0}}{ \mu_m\mu_n -\mu_j }\,,  \  \psi_{j,2}^{m|n} =  \frac{q_{j,2}^{m|n}}{ \mu_m\lambda_n -\mu_j }
\ \text{ and } \ \psi_{j,2}^{0|m,n} =  \frac{q_{j,2}^{0|m,n}}{ \lambda_m\lambda_n -\mu_j } \,.
\end{equation}

Next, to determine the Birkhoff invariants $\alpha_{jk}$ we consider the terms of degree $3$. Indeed, observe that since
$$
u_i \xi_i =  (\lambda_i + \sum_{\ell}  \alpha_{i\ell} \xi_\ell \eta_\ell ) \xi_i \quad  \text{ and } \quad  v_i \eta_i =  (\mu_i + \sum_{\ell}  \alpha_{i\ell} \xi_\ell \eta_\ell ) \eta_i \,,
$$
the coefficient of $\xi_k\xi_j\eta_k$ in  $\phi_j (u\xi, v\eta)$ is 
$$\alpha_{jk} + \phi_{j,3}^{k,j|k} \lambda_k\lambda_j\mu_k\,.$$
In fact, the functions $u\xi$ and $v\eta$ contain only monomials of degree $1$ and $3$, hence the quadratic part of $\phi_j$ does not contribute to the coefficients
of $\xi_k\xi_j\eta_k$ (since it has  degree $3$), the linear part of $\phi_j$, equal to $\xi_j$, contributes the coefficient of $\xi_k\xi_j\eta_k$  in the product
$u_j \xi_j$, that is $\alpha_{kj}$, and finally the cubic part of $\phi_j$ contributes only when evaluated on the linear parts of $u\xi$ and $v\eta$, hence  it contributes
the term $\phi_{j,3}^{k, j | k    } \lambda_k \lambda_j \mu_k$. 

The coefficient of $\xi_k\xi_j\eta_k$ in $p_j(\phi(\xi,\eta),\psi(\xi, \eta))$ is 
\begin{eqnarray*} 
\lambda_j \phi_{3,j}^{k,j|k} &+& \sum_{n=1}^d \left( p_{j,2}^{k,n|0} \phi_{n,2}^{j|k}  + p_{j,2}^{n,j |0} \phi_{n,2}^{k|k} +  p_{j,2}^{k|n}\psi_{n,2}^{j|k} + p_{j,2}^{n|k}\phi_{n,2}^{k,j|0} \right)
\\ &+&   \sum_{n=1}^d  (p_{j,2}^{0|n, k} +  p_{j,2}^{0|k, n}     )\psi_{n,2}^{k,j|0}    + p_{j,3}^{k,j|k} \,.
\end{eqnarray*} 
In fact, we have the following. The linear term of $p_j$ contributes the coefficient of $\xi_k\xi_j\eta_k$  in $\lambda_j \phi_j$, that is, the term 
$ \lambda_j \phi_{j,3}^{k,j | k}$, the cubic terms of $p_j$ contribute only the term $p_{j,3}^{k,j|k}$, since the only linear term of $\phi_i(\xi, \eta)$ is $\xi_i$ 
and of $\psi_i(\xi, \eta)$ is $\eta_i$, and all other terms come from the quadratic part of $p_j$.  

We analyze the contribution of $p_2(\phi(\xi, \eta), \psi(\xi, \eta))$ by  distinguishing $3$ cases;
the terms of the forms:
\begin{itemize}
 \item[(I)]$p_{j,2}^{m,n|0} \phi_m(\xi, \eta) \phi_n(\xi,\eta),$
 \item[(II)] $p_{j,2}^{m|n} \phi_m(\xi, \eta) \psi_n(\xi,\eta),$
\item[(III)] $p_{j,2}^{0|m,n} \psi_m(\xi, \eta) \psi_n(\xi,\eta).$
\end{itemize}

 The ordered monomial $\xi_k \xi_j \eta_k$ can be written as product of two factors in the following $4$ ways (there are $3$ permutations of $\xi_k \xi_j\eta_k$
which preserve the order of $\xi_k\xi_j$, and  the permutation $ \xi_k \xi_j \eta_k$ can be written as product in $2$ ways, see $a)$ and $c)$ below):
$$
a) \, \xi_k \cdot (\xi_j \eta_k)\,, \quad   b) \, (\xi_k  \eta_k) \cdot \xi_j , \quad   c) \,  (\xi_k  \xi_j) \cdot \eta_k \,, \quad   d)\, \eta_k \cdot (\xi_k \xi_j) \,.
$$
We recall that the only linear term of $\phi_i$ and $\psi_i$ are respectively $\xi_i$ and $\eta_i$. 

Therefore in case (I) we have that  one term of the form  
$a)$ for $m=k$ and one of the form $b)$ for $n=j$, namely
$$
p_{j,2}^{k,n|0} \phi_{n,2} ^{j | k}     \qquad \text{ and, respectively,} \qquad   p_{j,2}^{m,j|0}\phi_{m,2} ^{k| k} 
$$
and no terms of the form $c)$ or $d)$;  in case (II) we have one term of the form $a)$ for $m=k$
and one of the form $c)$ for $n=k$, namely
$$
p_{j,2}^{k|n} \psi_{n,2} ^{j | k}     \qquad \text{ and, respectively,} \qquad   p_{j,2}^{m|k}\phi_{m,2} ^{k,j| 0} 
$$
and no terms of the form $a)$ or $b)$; finally,  in case (III) we have one term of the form
$c)$ for $n=k$ and one of the form $d)$ for $m=k$, namely
$$
p_{j,2}^{0|m,k} \psi_{m,2} ^{k,j | 0}     \qquad \text{ and, respectively,} \qquad   p_{j,2}^{0|k,n}\psi_{n,2} ^{k,j| 0}. 
$$

Since $\lambda_k\mu_k=1$, we can express $\alpha_{kj}$ (or equivalently $b_{mn}$) in terms of the third orders jets of $ p_j$ and $q_j$. We have
\begin{equation}\label{eq:higheralpha}
\begin{aligned} 
\alpha_{jk} &=  \sum_{n=1}^d \left( p_{j,2}^{k,n|0} \phi_{n,2}^{j|k}  + p_{j,2}^{n,j |0} \phi_{n,2}^{k|k} +  p_{j,2}^{k|n}\psi_{n,2}^{j|k} + p_{j,2}^{n|k}\phi_{n,2}^{k,j|0} \right)
\\ &+   \sum_{n=1}^d  (p_{j,2}^{0|n, k} +  p_{j,2}^{0|k, n}     )\psi_{n,2}^{k,j|0}    + p_{j,3}^{k,j|k} \,.
\end{aligned} 
\end{equation}
We recall that the coefficients of the homogeneous polynomials  $\phi_{j,2}$ and $\psi_{j,2}$ are given in terms of the second order jets of $p_j$ and $q_j$  in formulas
\eqref{eq:phi_{j,2}} and \eqref{eq:psi_{j,2}}. 

\subsection{Twist conditions}\label{subsectwist} By following the proof of \cite[Theorem 1.4]{EFK}, one sees that the KAM theorem can be applied when the frequency vector $\omega$ of $Df(0)$ is irrational (in the sense that its coordinates are rationally independent as in \cite{EFK}) and the Birkhoff normal form  $B(r)=2\pi\langle\omega,r\rangle+\sum\frac{1}{2}b_{mn}r_mr_n+O_2(r)$ satisfies the following twist condition: the image of the \emph{frequency map} 
$$(r_1,\dots, r_d) \mapsto (\omega_1+\sum b_{1m}r_m, \dots, \omega_d+\sum b_{dm} r_m)$$ 
is not contained in a hyperplane. 
 \providecommand{\bysame}{\leavevmode\hbox to3em{\hrulefill}\thinspace}
\providecommand{\MR}{\relax\ifhmode\unskip\space\fi MR }
\providecommand{\MRhref}[2]{%
  \href{http://www.ams.org/mathscinet-getitem?mr=#1}{#2}
}
\providecommand{\href}[2]{#2}

\end{document}